\documentclass[a4paper,DIVcalc,12pt]{amsart}
\usepackage[centering,total={14.5cm,24cm}]{geometry}
\usepackage[onehalfspacing]{setspace}

\usepackage[T1]{fontenc}
\usepackage[latin1]{inputenc}
\usepackage{lmodern}
\usepackage[english]{babel}

\usepackage{amsmath}
\usepackage{amssymb}
\usepackage{amsthm}
\usepackage{mathrsfs}

\usepackage{braket}
\usepackage{dsfont}
\usepackage{faktor}
\usepackage{rotating}
\usepackage[arrow, matrix, curve]{xy}
\usepackage{mathtools}
\usepackage{hhline}
\usepackage{colortbl}
\usepackage{bbm}
\usepackage{stmaryrd}
\usepackage{paralist}
\usepackage{setspace}
\usepackage{hyperref}

\usepackage{tikz}

\usetikzlibrary{arrows,patterns,decorations.pathmorphing}

\tikzset{>=latex}
\usepackage{tkz-euclide}
\usetikzlibrary{shapes.misc}
\tikzset{cross/.style={cross out, draw=black, fill=none, minimum size=2*(#1-\pgflinewidth), inner sep=0pt, outer sep=0pt}, cross/.default={2pt}}

\DeclareFontFamily{OT1}{pzc}{}
\DeclareFontShape{OT1}{pzc}{m}{it}{<-> s * [1.10] pzcmi7t}{}
\DeclareMathAlphabet{\mathpzc}{OT1}{pzc}{m}{it}

\DeclareFontFamily{U}{mathx}{\hyphenchar\font45}
\DeclareFontShape{U}{mathx}{m}{n}{<-> mathx10}{}
\DeclareSymbolFont{mathx}{U}{mathx}{m}{n}
\DeclareMathAccent{\widebar}{0}{mathx}{"73}

\vfuzz \hfuzz
\newtheorem{prop}{Proposition}[section]
\newtheorem{thm}[prop]{Theorem}
\newtheorem{cor}[prop]{Corollary}
\newtheorem{lem}[prop]{Lemma}

\newtheorem{thmintro}{Theorem}

\theoremstyle{definition}
\newtheorem{defi}[prop]{Definition}
\newtheorem{con}[prop]{Construction}
\newtheorem{expl}[prop]{Example}

\newtheorem*{defiintro}{Definition}

\theoremstyle{remark}
\newtheorem{rem}[prop]{Remark}

\newtheorem*{remintro}{Remark}

\begin{document}
\numberwithin{figure}{section}
\numberwithin{table}{section}

\title[Theta functions, broken lines and $2$-marked log GW]{Theta functions, broken lines and $2$-marked log Gromov-Witten invariants}

\author{Tim Graefnitz} 
\address{\tiny Tim Gr\"afnitz, University of Cambridge, DPMMS, Wilberforce Road, Cambridge, CB3 0WB, UK}
\email{tim.graefnitz@gmx.de}

\begin{abstract}
Theta functions were defined for varieties with effective anticanonical divisor \cite{GHS} and are related to certain punctured Gromov-Witten invariants \cite{ACGS2}. In this paper we show that in the case of a log Calabi-Yau surface $(X,D)$ with smooth very ample anticanonical divisor we can circumvent the notion of punctured Gromov-Witten invariants and relate theta functions and their multiplicative structure to certain $2$-marked log Gromov-Witten invariants. This is a natural extension of the correspondence between wall functions and $1$-marked log Gromov-Witten invariants \cite{Gra}.
\end{abstract}

\maketitle

\setcounter{tocdepth}{1}

\renewcommand{\baselinestretch}{0.9}\normalsize
\tableofcontents
\renewcommand{\baselinestretch}{1.25}\normalsize

\section*{Introduction}									

Classically theta functions were defined in the theory of elliptic curves. Mumford generalized their definition to Abelian varieties and studied degenerations of such varieties \cite{Mum}. In \cite{GHK} theta functions were used to construct the mirror to certain maximally degenerated log Calabi-Yau pairs. They naturally fit within the Gross-Siebert program of mirror symmetry \cite{GSdataI}\cite{GSdataII}\cite{GSreconstruction}. It was shown in \cite{GScanonical}, Theorem 4.5, that these theta functions are linked to certain \textit{punctured Gromov-Witten invariats} \cite{ACGS2}. This leads to an \textit{intrinsic mirror symmetry} construction \cite{GSintrinsic}, where the mirror to a Calabi-Yau variety or log Calabi-Yau pair is constructed from its theta functions or, equivalently, its punctured invariants.

In this paper we consider the case complementary to \cite{GHK} of a log Calabi-Yau pair $(X,D)$ with $D$ a \textit{smooth} divisor. We restrict to the case of a variety $X$ with very ample anticanonical bundle to obtain a combinatorial description. Moreover, we restrict to the $2$-dimensional case. We show that in this case we can circumvent the notion of punctured invariants and relate theta functions to certain $2$-marked \textit{log Gromov-Witten invariants}.

\begin{defiintro}
For an effective curve class $\beta$ on $X$ and $p,q\in\mathbb{Z}_{>0}$ with $p+q=D\cdot\beta$ let $\beta_{p,q}$ be the class of stable log maps of class $\beta$ with two marked points of contact orders $p$ and $q$ with $D$. Let $\mathscr{M}(X,\beta_{p,q})$ be the moduli space of basic stable log maps of class $\beta_{p,q}$. By \cite{GSloggw} this a proper Deligne-Mumford stack of virtual dimension $1$ and admits a virtual fundamental class $\llbracket\mathscr{M}(X,\beta_{p,q})\rrbracket$. Let $\textup{ev} : \mathscr{M}(X,\beta_{p,q}) \rightarrow D$ be the evaluation map at the marked point of order $p$. Define the $2$-marked log Gromov-Witten invariant
\[ N_{p,q}(X,\beta) := \int_{\llbracket\mathscr{M}(X,\beta_{p,q})\rrbracket} \textup{ev}^\star[\textup{pt}] \in \mathbb{Q} \]
and write
\[ N_{p,q}(X) := \sum_\beta N_{p,q}(X,\beta), \]
where the sum is over all effective curve classes of $X$. We will omit $X$ in the notation whenever it is clear from the context. 
\end{defiintro}

For an asymptotic exponent $m$ and a point $P$ on the dual intersection complex $(B,\mathscr{P},\varphi)$ of $X$, there is a theta function defined by
\[ \vartheta_m(P) = \sum_{\mathfrak{b}\in\mathfrak{B}_m(P)} a_{\mathfrak{b}} z^{m_{\mathfrak{b}}} \]
The sum is over \textit{broken lines} -- piecewise linear maps $\mathfrak{b} : [0,\infty) \rightarrow B$ with a monomial $a_iz^{m_i}$ attached to every linear segment such that $m_i=(\widebar{m}_i,h_i)$ with $\widebar{m}_i$ parallel to the segment and $a_iz^{m_i}$, $a_{i+1}z^{m_{i+1}}$ are related by wall crossing in $\mathscr{S}$.

By \cite{GHS}, Theorem 3.24, theta functions generate a ring with multiplication rule
\[ \vartheta_p(P) \cdot \vartheta_q(P) = \sum_{r=0}^\infty \alpha_{p,q}^r(P) \vartheta_r(P) \]
with structure constants
\[ \alpha_{p,q}^r(P) = \sum_{\substack{(\mathfrak{b}_1,\mathfrak{b}_2)\in\mathfrak{B}_p(P)\times\mathfrak{B}_q(P) \\ m_{\mathfrak{b}_1}+m_{\mathfrak{b}_2}= \ m}} a_{\mathfrak{b}_1}a_{\mathfrak{b}_2}. \]

Let $P$ be a point infinitely far away from the bounded maximal cell of $B$. Let $m_{\textup{out}}$ be the unique unbounded direction of $B$ and write $x=z^{(-m_{\textup{out}},-1)}$ and $t=z^{(0,1)}$. Note that at $P$ we have $\varphi(-m_{\textup{out}})=-1$, so $x$ has $t$-order zero.

\begin{thmintro}
\label{thm:main}
We have
\[ \vartheta_q(P) = x^{-q} + \sum_{p=1}^\infty pN_{p,q}t^{p+q}x^p. \]
\end{thmintro}

The reason for this equation is as follows. As we move $P$ away from the bounded maximal cell of $B$ the slope of the walls (meaning their scalar product with the unbounded direction) increases. If we move sufficiently far away all broken lines ending in $P$ have to be parallel to the unbounded direction (Proposition \ref{prop:parallel}). We can complete such a broken line to a tropical curve with two unbounded legs. One of these legs contains $P$, corresponding to a fixed point. Then the tropical correspondence theorem for log Calabi-Yau pairs with smooth divisor \cite{Gra} (more precisely, an extension of it incluing point conditions) gives a relation between broken lines and $2$-marked log Gromov-Witten invariants, leading to the above correspondence for theta functions. By a similar reason we have the following.

\begin{thmintro}
\label{thm:main2}
The multiplication rule for theta functions is given by
\[ \vartheta_p \cdot \vartheta_q = \vartheta_{p+q} + \sum_{r=0}^{\text{max}\{p,q\}-1} \alpha_{p,q}^r \vartheta_r \]
with structure constants (we define $N_{p,q}(X)=0$ when $p \leq 0$ or $q\leq 0$)
\[ \alpha_{p,q}^r = \left((p-r)N_{p-r,q} + (q-r)N_{q-r,p}\right)t^{p+q-r}. \]
Comparing this with Theorem \ref{thm:main} we obtain relations between the numbers $N_{p,q}$. These determine all $N_{p,q}$ from the invariants $N_{1,q}$ with fixed point multiplicity $1$.
\end{thmintro}

\begin{remintro}
By \cite{GRZ}, Corollary 6.6, the invariants $N_{1,q}(X,\beta)$ agree with $1$-marked invariants $N_q(\hat{X},\pi^\star\beta-C)$ of the blow up at a point $\pi : \hat{X} \rightarrow X$, where $C$ is the exceptional divisor. Using this relation, Theorem \ref{thm:main2} above, the log-local correspondence of \cite{GGR} and results of \cite{LLW} we show in \cite{GRZ} that $\vartheta_1$ agrees with the open mirror map of outer Aganagic-Vafa branes in framing zero.
\end{remintro}

\begin{remintro}
In \cite{GSintrinsic} the structure constants were defined as
\[ \alpha_{p,q}^r = \sum_\beta N_{pq}^r(\beta)t^\beta \]
where the sum is over all effective curve classes $\beta$ of $X$ and $N_{pq}^r(\beta)$ are certain punctured Gromov-Witten invariants of class $\beta$. In particular $\alpha_{m_1,m_2}^m$ does not depend on $P$. It was shown in \cite{GScanonical}, Theorem 4.5, that this agrees with the above definition of $\alpha_{p,q}^r(P)$. In combination with Theorem \ref{thm:main2} we obtain a triangle of relations
\begin{equation*}
\begin{xy}
\xymatrix{
\alpha_{p,q}^r \ar[rr]^{\textup{\cite{GScanonical}, Theorem 4.5}} \ar[rd]_{\textup{Theorem \ref{thm:main2}}} 		& & N_{p,q}^r \ar[ll]\ar[ld] \\
					& N_{p,q} \ar[lu]\ar[ru] &
}
\end{xy}
\end{equation*}
\end{remintro}

\subsection*{Acknowledgements}
The author received support from the DFG SFB/TR 191 and the ERC Advanced Grant MSAG.

\section{Toric degenerations and wall structures}					

Tropical geometry is a piecewise linear version of algebraic geometry, hence naturally linked to combinatorics. The most natural class of varieties on which to consider tropical geometry are toric varieties. They admit a combinatorial description in terms of the orbits of their torus action. This can be made explicit in terms of a fan or, dually and given a polarization, a ``momentum'' polytope. A similar combinatorial description for a non-toric variety $X$ is obtained via a toric degeneration $\mathfrak{X}$ -- a degeneration whose central fiber $X_0$ is a union of toric varieties glued along toric divisors and such that the family is strictly semistable away from a codimension $2$ subset. See \cite{GSdataI} for more details.

The dual intersection complex $(B,\mathscr{P})$ of $\mathfrak{X}$ is obtained by gluing together the fans of the irreducible components of $X_0$ according to the combinatorics of their intersection. Locally at a vertex such a fan induces an affine structure on $B$. An affine structure on maximal cells is given by the local structure of the family $\mathfrak{X}$ at the corresponding point. This is an affine toric variety defined by a cone over a polytope, and this polytope gives the affine structure on the maximal cell. These affine structures need not fit together, so $B$ is an affine manifolds with singularities. It is glued from (possibly unbounded) polytopes and $\mathscr{P}$ is a collection of these. A polarization of $X$ corresponds to a strictly convex piecewise affine function $\varphi$ on $(B,\mathscr{P})$. 

Similarly, the intersection complex $(\check{B},\check{\mathscr{P}},\check{\varphi})$ is obtained by gluing together momentum polytopes of the irreducible components of $X_0$. To do so, we need a polarization on $X$. In this case the piecewise affine function $\check{\varphi}$ describes the family $\mathfrak{X}$ locally: it is the upper convex hull of $\check{\varphi}$.

\begin{expl}
A toric degeneration of $(\mathbb{P}^2,E)$, for $E$ an elliptic curve, is given by
\vspace{-3mm}
\begin{eqnarray*}
\mathfrak{X} &=& \{XYZ=t^3(W+f_3)\} \subset \mathbb{P}(1,1,1,3)\times\mathbb{A}^1 \rightarrow \mathbb{A}^1, \\
\mathfrak{D} &=& \{W=0\} \subset \mathfrak{X}.
\end{eqnarray*}
Here $W$ is the variable in $\mathbb{P}(1,1,1,3)$ of weight $3$, $t$ is the variable of $\mathbb{A}^1$ and the map is given by projection to $t$. Moreover, $f_3$ is a polynomial in $X,Y,Z$, general in the sense that $XYZ/t^3-f_3$ is nonsingular for general $t$.

Indeed, for $t\neq 0$ we can resolve for $W=XYZ/t^3-f_3$, so the general fiber of $\mathfrak{X}$ is $X=\mathbb{P}^2$ and the general fiber of $\mathfrak{D}$ is an elliptic curve $E$ by the generality condition on $f_3$. For $t=0$ we have $XYZ=0$, so $X_0$ is a union of three $\mathbb{P}(1,1,3)$ glued along toric divisors. Figure \ref{fig:P2} shows the intersection complex and its dual for this toric degeneration.
\end{expl}

\begin{figure}[h!]
\centering
\begin{tikzpicture}[scale=1.9]
\coordinate[fill,circle,inner sep=1.2pt,label=below:${W}$] (0) at (0,0);
\coordinate[fill,circle,inner sep=1.2pt,label=below:${Z}$] (1) at (-1,-1);
\coordinate[fill,circle,inner sep=1.2pt,label=below:${X}$] (2) at (2,-1);
\coordinate[fill,circle,inner sep=1.2pt,label=above:${Y}$] (3) at (-1,2);
\coordinate[fill,cross,inner sep=3pt,rotate=45] (1a) at (-0.5,-0.5);
\coordinate[fill,cross,inner sep=3pt,rotate=63.43] (2a) at (1,-0.5);
\coordinate[fill,cross,inner sep=3pt,rotate=26.57] (3a) at (-0.5,1);
\draw (1) -- (2) -- (3) -- (1);
\draw (0) -- (1a);
\draw (0) -- (2a);
\draw (0) -- (3a);
\draw[dashed] (1a) -- (1);
\draw[dashed] (2a) -- (2);
\draw[dashed] (3a) -- (3);
\coordinate[fill,circle,inner sep=1.2pt] (a) at (0,-1);
\coordinate[fill,circle,inner sep=1.2pt] (b) at (1,-1);
\coordinate[fill,circle,inner sep=1.2pt] (c) at (-1,0);
\coordinate[fill,circle,inner sep=1.2pt] (d) at (-1,1);
\coordinate[fill,circle,inner sep=1.2pt] (e) at (1,0);
\coordinate[fill,circle,inner sep=1.2pt] (f) at (0,1);
\draw[->] (0.5,1.5) node[right]{$\varphi=1$} to [bend right=10] (-0.2,1.2);
\draw[->] (0.2,0.2) node[right]{$\varphi=0$} to [bend right=10] (0,0);
\draw[<->] (2,0.5) -- (2.8,0.5);
\end{tikzpicture}
\begin{tikzpicture}[scale=1.3]
\coordinate[fill,circle,inner sep=1.2pt] (1) at (1,0);
\coordinate[fill,circle,inner sep=1.2pt] (2) at (0,1);
\coordinate[fill,circle,inner sep=1.2pt] (3) at (-1,-1);
\coordinate (1a) at (2.5,0);
\coordinate (2a) at (0,2.5);
\coordinate (3a) at (-2.5,-2.5);
\coordinate (12) at (0.5,0.5);
\coordinate (12a) at (0.5,2.5);
\coordinate (12b) at (2.5,0.5);
\coordinate (23) at (-0.5,0);
\coordinate (23a) at (-0.5,2.5);
\coordinate (23b) at (-2.5,-2);
\coordinate (31) at (0,-0.5);
\coordinate (31a) at (2.5,-0.5);
\coordinate (31b) at (-2,-2.5);
\draw (1) -- (2) -- (3) -- (1);
\draw (1) -- (1a);
\draw (2) -- (2a);
\draw (3) -- (3a);
\fill[opacity=0.2] (12) -- (12a) -- (12b) -- (12);
\fill[opacity=0.2] (23) -- (23a) -- (23b) -- (23);
\fill[opacity=0.2] (31) -- (31a) -- (31b) -- (31);
\draw[dashed] (12a) -- (12) -- (12b);
\draw[dashed] (23a) -- (23) -- (23b);
\draw[dashed] (31a) -- (31) -- (31b);
\draw (12) node[fill,cross,inner sep=1pt,rotate=45]{};
\draw (23) node[fill,cross,inner sep=1pt,rotate=63.43]{};
\draw (31) node[fill,cross,inner sep=1pt,rotate=26.57]{};
\draw (0.1,0) node{$\varphi=0$};
\coordinate[fill,circle,inner sep=1.2pt,label=below:${\varphi=1}$] (1c) at (2,0);
\coordinate[fill,circle,inner sep=1.2pt,label=below:${\varphi=1}$] (2c) at (0,2);
\coordinate[fill,circle,inner sep=1.2pt,label=below:${\varphi=1}$] (3c) at (-2,-2);
\end{tikzpicture}
\caption{The intersection complex (left) of $(\mathbb{P}^2,E)$ and its dual (right).}
\label{fig:P2}
\end{figure}
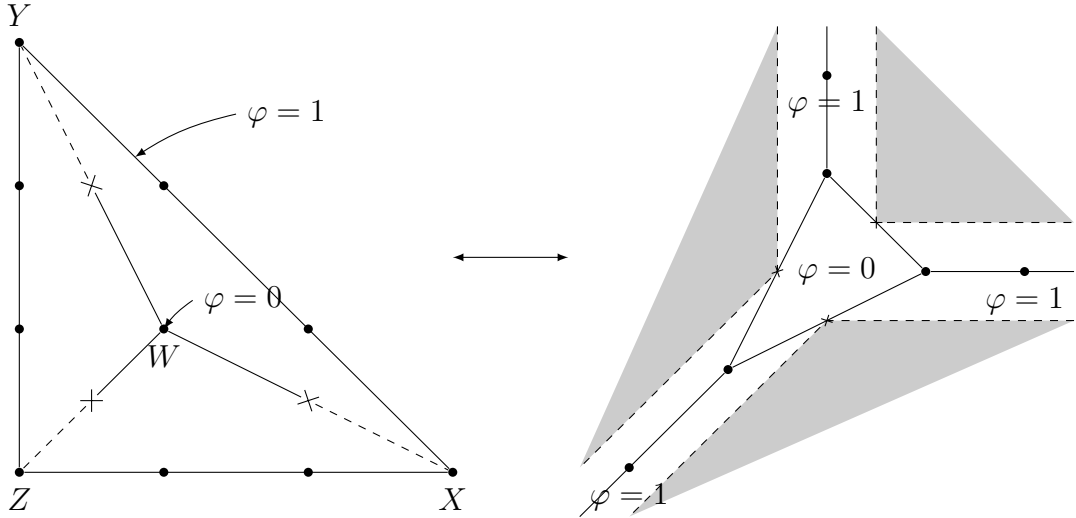

It is natural to ask whether one can go back and construct a toric degeneration $\check{\mathfrak{X}}$ from a triple $(B,\mathscr{P},\varphi)$ as above such that $(B,\mathscr{P},\varphi)$ is the intersection complex of $\check{\mathfrak{X}}$. This was solved under some assumptions in \cite{GSreconstruction} and involves the iterative construction of wall structure via scattering calculations. A \textit{wall} is a codimension $1$ polyhedral subset $\mathfrak{p}$ of $B$ with some attached function $f_{\mathfrak{p}}$ living in the ring $R_\varphi$ defined by
\begin{eqnarray*}
P_\varphi &=& \{m=(\widebar{m},h)\in M\oplus\mathbb{Z} \ | \ h \geq \varphi(\widebar{m})\}, \\
R_\varphi &=& \varprojlim\mathbb{C}[P_\varphi]/(t^k), \quad t=z^{(0,1)}.
\end{eqnarray*}
A wall defines an automorphism of some localization of $R_\varphi$ by
\[ \theta_{\mathfrak{p}}(z^{(\widebar{m},h)}) = f_{\mathfrak{p}}^{-\braket{n_{\mathfrak{p}},\widebar{m}}}z^{(\widebar{m},h)}. \]
Here $n_{\mathfrak{p}}$ is a primitive normal vector to $\mathfrak{p}$ (this involves a choice of orientation). We restrict to the $2$-dimensional case such that walls are $1$-dimensional.

The initial wall structure $\mathscr{S}_0$ consists of walls coming out of the affine singularities of $B$ with attached functions $1+z^{(\widebar{m},0)}$, where $\widebar{m}$ is the direction of $\mathfrak{p}$. Scattering means whenever two or more walls intersect they produce more walls with base at the intersection, such that the clockwise composition of the automorphisms $\theta_{\mathfrak{p}}$ is the identity. For any finite $t$-order $k$ this produces finitely many new rays, leading to a wall structure $\mathscr{S}_k$ that is ``consistent to order $k$''. The formal limit for $k\rightarrow\infty$ is denoted by $\mathscr{S}_\infty$. Note that the $t$-order of $z^{(\widebar{m},h)}$ at $x$ is $\varphi_x(-\widebar{m})+h\geq 0$, where $\varphi_x$ is a local representative of $\varphi$ at $x$, since
\[ z^{(\widebar{m},h)} = \left(z^{(-\widebar{m},\varphi_x(-\widebar{m}))}\right)^{-1} t^{\varphi_x(-\widebar{m})+h}. \]

\begin{expl}
Figure \ref{fig:P2wall} shows the wall structure consistent to $t$-order $2\cdot 3=6$ for $(\mathbb{P}^2,E)$ in two different charts of the affine structure. There are functions attached to each of these walls, but we don't show them for readability.
\end{expl}

\begin{figure}[h!]
\centering
\begin{tikzpicture}[scale=1.4]
\coordinate[fill,circle,inner sep=1.2pt] (1) at (1,0);
\coordinate[fill,circle,inner sep=1.2pt] (2) at (0,1);
\coordinate[fill,circle,inner sep=1.2pt] (3) at (-1,-1);
\coordinate (1a) at (2.5,0);
\coordinate (2a) at (0,2.5);
\coordinate (3a) at (-2.5,-2.5);
\coordinate (12) at (0.5,0.5);
\coordinate (12a) at (0.5,2.5);
\coordinate (12b) at (2.5,0.5);
\coordinate (23) at (-0.5,0);
\coordinate (23a) at (-0.5,2.5);
\coordinate (23b) at (-2.5,-2);
\coordinate (31) at (0,-0.5);
\coordinate (31a) at (2.5,-0.5);
\coordinate (31b) at (-2,-2.5);
\draw (1) -- (2) -- (3) -- (1);
\draw (1) -- (1a);
\draw (2) -- (2a);
\draw (3) -- (3a);
\fill[opacity=0.2] (12) -- (12a) -- (12b) -- (12);
\fill[opacity=0.2] (23) -- (23a) -- (23b) -- (23);
\fill[opacity=0.2] (31) -- (31a) -- (31b) -- (31);
\draw[dashed] (12a) -- (12) -- (12b);
\draw[dashed] (23a) -- (23) -- (23b);
\draw[dashed] (31a) -- (31) -- (31b);
\draw (12) node[fill,cross,inner sep=1pt,rotate=45]{};
\draw (23) node[fill,cross,inner sep=1pt,rotate=63.43]{};
\draw (31) node[fill,cross,inner sep=1pt,rotate=26.57]{};
\draw (0.5,0.5) -- (-0.5,1.5) -- (-0.3,2.5);
\draw (0.5,0.5) -- (1.5,-0.5) -- (2.5,-0.3);
\draw (0,-0.5) -- (2,0.5) -- (2.5,0.375);
\draw (0,-0.5) -- (-2,-1.5) -- (-2.5,-2.125);
\draw (-0.5,0) -- (0.5,2) -- (0.375,2.5);
\draw (-0.5,0) -- (-1.5,-2) -- (-2.125,-2.5);
\draw (1,0) -- (2.5,0.3);
\draw (1,0) -- (2.5,-0.375);
\draw (0,1) -- (0.3,2.5);
\draw (0,1) -- (-0.375,2.5);
\draw (-1,-1) -- (-2.5,-2.2);
\draw (-1,-1) -- (-2.2,-2.5);
\draw (3,0) -- (3.5,0);
\draw (3,-0.1) -- (3.5,-0.1);
\draw (0,-3);
\end{tikzpicture}
\begin{tikzpicture}[scale=1.2,rotate=90]
\coordinate[label=right:${(1,0)}$,fill,circle,inner sep = 1pt] (1) at (0,0);
\coordinate[label=left:${(0,0)}$,fill,circle,inner sep = 1pt] (2) at (0,1);
\coordinate[label=below:${(2,-3)}$,fill,circle,inner sep = 1pt] (3) at (-3,-1);
\coordinate[label=below:${(-1,-3)}$,fill,circle,inner sep = 1pt] (4) at (-3,2);
\draw (3) -- (1) -- (2) -- (4);
\draw (1) -- (3,0);
\draw (2) -- (3,1);
\draw[dashed] (3) -- (3,-1);
\draw[dashed] (4) -- (3,2);
\draw[dashed] (-1,0) -- (0,0.5) -- (-1,1);
\draw[dashed] (-3,-0.8) -- (-1.5,-0.5) -- (-1,0);
\draw[dashed] (-3,1.8) -- (-1.5,1.5) -- (-1,1);
\fill[opacity=0.2] (-3,1.8) -- (-1.5,1.5) -- (-1,1) -- (0,0.5) -- (-1,0) -- (-1.5,-0.5) -- (-3,-0.8);
\coordinate[fill,cross,line width=1pt,inner sep=2pt] (5) at (0,0.5);
\coordinate[fill,cross,line width=1pt,inner sep=2pt,rotate=26.57] (6) at (-1.5,-0.5);
\coordinate[fill,cross,line width=1pt,inner sep=2pt,rotate=-26.57] (7) at (-1.5,1.5);
\draw (0,0.5) -- (0,2);
\draw (0,0.5) -- (0,-1);
\draw (-1.5,-0.5) -- (3,1);
\draw (-1.5,1.5) -- (3,0);
\draw (-3,-1) -- (3,0);
\draw (-3,2) -- (3,1);
\draw (1.5,0.5) -- (3,0.5);
\draw (0,-0.5) -- (3,-0.5);
\draw (0,1.5) -- (3,1.5);
\draw (0,0) -- (3,0.5);
\draw (0,0) -- (3,-1);
\draw (0,1) -- (3,0.5);
\draw (0,1) -- (3,2);
\draw (-3,-1) -- (3,-1/3);
\draw (-3,2) -- (3,1+1/3);
\end{tikzpicture}
\caption{The wall structure for $(\mathbb{P}^2,E)$ of order $2\cdot 3=6$.}
\label{fig:P2wall}
\end{figure}
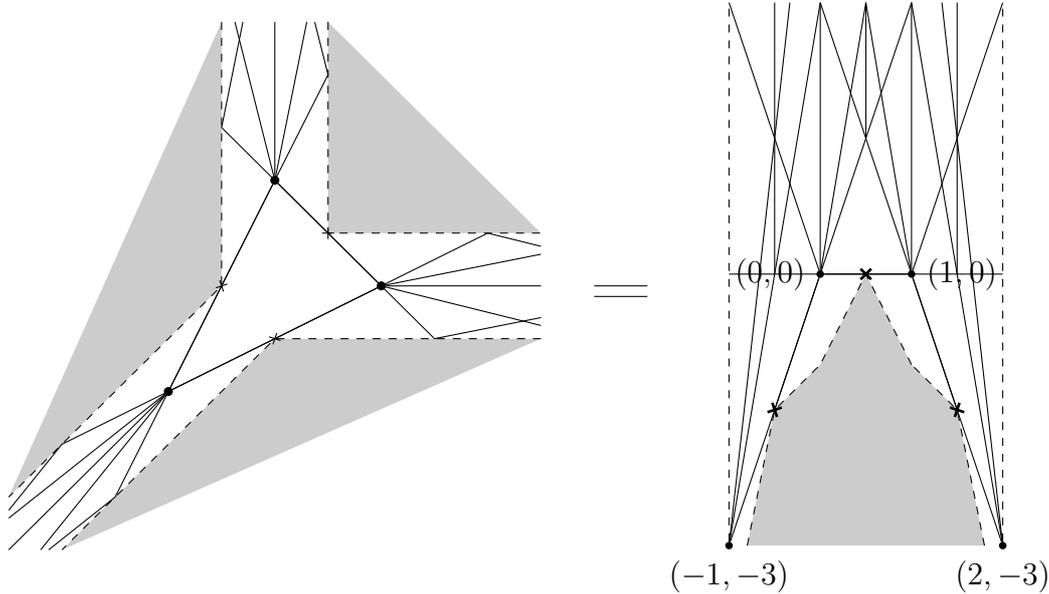

The idea of the Gross-Siebert program is that the family $\check{\mathfrak{X}}$ constructed this way is a degeneration of the mirror to $(X,D)$. As is well known, the mirror to a log Calabi-Yau pair $(X,D)$ should be a Landau-Ginzburg model, that is, a toric variety $Y$ together with a ''superpotential'' $W : Y \rightarrow \mathbb{C}$ whose critical locus is compact. It is shown in \cite{CPS} that such a superpotential can be constructed from broken lines, see \S\ref{S:broken}. For more details on these constructions see \cite{GSdataI}\cite{GSdataII}\cite{GSreconstruction}.

\section{Broken lines}								
\label{S:broken}

\begin{defi}
\label{defi:broken}
A \textit{broken line} for a wall structure $\mathscr{S}$ on $(B,\mathscr{P})$ is a proper continuous map
\[ \mathfrak{b} : (-\infty,0] \rightarrow B_0 \]
with image disjoint from any joints of $\mathscr{S}$, along with a sequence $-\infty=t_0<t_1<\cdots<t_r=0$ for some $r\geq 1$ with $\mathfrak{b}(t_i)\in|\mathscr{S}|$ for $i\leq r-1$, and for each $i=1,\ldots,r$ an expression $a_iz^{m_i}$ with $a_i\in \mathbb{C}\setminus\{0\}$, $\widebar{m}_i\in\Lambda_{\mathfrak{b}(t)}$ for any $t\in(t_{i-1},t_i)$, defined at all points of $\mathfrak{b}([t_{i-1},t_i])$, and subject to the following conditions:
\begin{compactenum}
\item $\mathfrak{b}|_{(t_{i-1},t_i)}$ is a non-constant affine map with image disjoint from $|\mathscr{S}|$, hence contained in the interior of a unique chamber $\mathfrak{u}_i$ of $\mathscr{S}$, and $\mathfrak{b}'(t)=-m_i$ for all $t\in(t_{i-1},t_i)$.
\item For each $i=1,\ldots,r-1$ the expression $a_{i+1}z^{m_{i+1}}$ is a result of transport of $a_iz^{m_i}$ from $\mathfrak{u}_i$ to $\mathfrak{u}_{i+1}$, i.e., is a term in the expansion of $\theta_{\mathfrak{u}_i\mathfrak{u}_{i+1}}(a_iz^{m_i})$.
\item $a_1=1$ and $m_1=(\widebar{m}_1,h)$ has $t$-order zero at $\mathfrak{b}(t_1)$, i.e., $h=\varphi(\widebar{m}_1)$.
\end{compactenum}
Write $a_{\mathfrak{b}}z^{m_{\mathfrak{b}}}$ for the ending monomial $a_rz^{m_r}$.
\end{defi}


Let $(B,\mathscr{P},\varphi)$ be the dual intersection complex of a pair $(X,D)$ consisting of a surface $X$ and a smooth very ample anticanonical divisor $D$. Then $B$ has one unbounded direction $m_{\textup{out}}$, since $D$ is smooth. Choose $\varphi$ such that $\varphi(m_{\textup{out}})=1$ on all unbounded cells. Let $\mathscr{S}_\infty$ be the consistent wall structure defined by $(B,\mathscr{P},\varphi)$. Then $\varphi(m_{\textup{out}})=1$ along all walls of $\mathscr{S}_\infty$, since walls are not contained in the interior of the bounded maximal cell (\cite{Gra}, Lemma 5.14). Hence, each broken line for $\mathscr{S}_\infty$ has asymptotic monomial $m_1=(qm_{\textup{out}},q)$ for some $q\in\mathbb{N}$.

\begin{defi}
For a point $P\in B_0$ let $\mathfrak{B}(P)$ be the set of be the set of broken lines for the consistent wall structure $\mathscr{S}_\infty$ with endpoint $\mathfrak{b}(0)=P$. 

For $P\in B_0$ and $q\in\mathbb{N}$ let $\mathfrak{B}_q(P)$ be the set of broken lines in $\mathfrak{B}(P)$ with asymptotic monomial $m_1=(qm_{\textup{out}},q)$. 

Write $\mathfrak{B}_q^{(k)}(P)$ for the subset of $\mathfrak{B}_q(P)$ consisting of walls such that the ending monomial $a_{\mathfrak{b}}z^{m_{\mathfrak{b}}}$ has $t$-order $\leq k$. This means if $m_{\mathfrak{b}}=(\widebar{m},h)$, then $\varphi(-\widebar{m}) \leq k-h$. 

Note that broken lines in $\mathfrak{B}^{(k)}_q(P)$ only break at walls of $\mathscr{S}_k$.
\end{defi}

\begin{defi}
For a chamber $\mathfrak{u}$ of $\mathscr{S}_k$, define the superpotential to order $k$ by
\[ W^k(P) = \sum_{\mathfrak{b}\in\mathfrak{B}^{(k)}(P)} a_{\mathfrak{b}}z^{m_{\mathfrak{b}}} \]
\end{defi}

\begin{prop}[\cite{CPS}, Lemma 4.7]
\label{prop:indep}
For a given order $k$ and chamber $\mathfrak{u}$ of $\mathscr{S}_k$, the superpotential $W^k(P)$ is the same for all general (not contained in a certain lower-dimensional subset) points $P\in\mathfrak{u}$.
\end{prop}

\begin{prop}[\cite{GRZ}, Proposition 3.5]
\label{prop:parallel}
Let $\mathfrak{b}$ be a broken line in $\mathfrak{B}_q^{(k)}(P)$. If $P$ lies in an unbounded chamber of $\mathscr{S}_k$, then $\widebar{m}_{\mathfrak{b}}$ is parallel to $m_{\textup{out}}$.
\end{prop}

\begin{proof} 
Let $P$ be a point in an unbounded chamber $\mathfrak{u}$ of $\mathscr{S}_k$. By Proposition \ref{prop:indep} we can assume $P$ is ``near infinity'', i.e., arbitrarily far away from the bounded maximal cell. Near infinity, all walls of $\mathscr{S}_k$ are parallel: Walls that are not parallel cross the unbounded edges of $B$ arbitrarily many times, increasing the $t$-order each time, in contradiction with the wall being of order $\leq k$.

Let $\mathfrak{b}$ be a broken line ending in $P$. We say $\mathfrak{b}$ ``breaks'' if it interacts nontrivially with a wall. If $\mathfrak{b}$ does not break at all, then $m_{\mathfrak{b}}$ is an asymptotic direction, hence parallel to $m_{\text{out}}$. So assume $\mathfrak{b}$ breaks at least once. Then it has to break at least once with a wall $\mathfrak{p}$ that is not parallel to $m_{\text{out}}$, since the direction of the ``incoming'' unbounded line segment is parallel to $m_{\text{out}}$ and parallel lines don't intersect. Let $P'\in\mathfrak{p}$ be the point where this breaking happens. Then $P'$ is ``close'' to the bounded maximal cell, meaning it cannot be arbitrarily far away from it where all walls are parallel to $m_{\text{out}}$. Hence, $P$ and $P'$ are arbitrarily far away from each other. The broken line $\mathfrak{b}$ has finitely many line segments. Otherwise it would have infinite order, since its order increases with each breaking. So to connect $P'$ with $P$ there must be a line segment of $\mathfrak{b}$ whose length in $m_{\text{out}}$-direction is arbitrarily large. This means it is either parallel to $m_{\text{out}}$ or has infinite order $h_{\mathfrak{b}}$. The latter cannot be, since we assumed $h_{\mathfrak{b}}\leq k$. Hence, this segment is parallel to $m_{\text{out}}$ and ends in the region ``near infinity'' where all rays are parallel. Since parallel lines don't intersect, it must be the last segment of $\mathfrak{b}$.
\end{proof}

\begin{defi}
Let $\mathfrak{B}_{p,q}(P)$ be the set of broken lines $\mathfrak{b}$ with endpoint $\mathfrak{b}(0)=P$, asymptotic direction $\widebar{m}_1=qm_{\textup{out}}$, and ending direction $\widebar{m}_{\mathfrak{b}}=-pm_{\textup{out}}$.
\end{defi}

\begin{prop}
\label{prop:p+q}
For a broken line $\mathfrak{b}$ in $\mathfrak{B}_{p,q}(P)$ we have $m_1=(qm_{\textup{out}},q)$ and $m_{\mathfrak{b}}=(-pm_{\textup{out}},q)$. In particular, the ending monomial $a_{\mathfrak{b}}z^{m_{\mathfrak{b}}}$ has $t$-order $p+q$.
\end{prop}

\begin{proof}
We have $m_1=(qm_{\textup{out}},q)$ because $z^{m_1}$ has $t$-order zero by the definition of broken lines. Since all walls in $\mathscr{S}_\infty$ are of the form $1+a_{\mathfrak{p}}z^{(m_{\mathfrak{p}},0)}$, i.e., with zero as second component of the exponent, the ending monomial of $\mathfrak{b}$ is $a_{\mathfrak{b}}z^{m_{\mathfrak{b}}}=a_{\mathfrak{b}}z^{(-pm_{\textup{out}},q)}$. Hence, the $t$-order of its ending monomial is $\varphi(pm_{\textup{out}})+q=p+q$. 
\end{proof}

\begin{cor}
\label{cor:broken}
If $P$ lies in an unbounded chamber of $\mathscr{S}_k$, then
\[ \mathfrak{B}_q^{(k)}(P) = \coprod_{p=1}^{k-q} \mathfrak{B}_{p,q}(P). \]
\end{cor}

\begin{proof}
By Proposition \ref{prop:parallel} we have $\mathfrak{B}_q^{(k)}(P) \subset \coprod_{p=1}^{\infty} \mathfrak{B}_{p,q}(P)$. By Proposition \ref{prop:p+q} we have $\mathfrak{B}_{p,q}(P) \subset \mathfrak{B}_q^{(k)}(P)$ if and only if $p+q\leq k$. Clearly, for different $p$ the sets $\mathfrak{B}_{p,q}(P)$ are disjoint.
\end{proof}

\begin{cor}
The set $\mathfrak{B}_{p,q}(P)$ is finite for all $p,q\in\mathbb{N}$.
\end{cor}

\begin{proof}
By Corollary \ref{cor:broken} we have $\mathfrak{B}_{p,q}(P) \subset \mathfrak{B}_q^{(k)}(P)$ for $k\geq p+q$. The set $\mathfrak{B}_q^{(k)}(P)$ is finite for all $k$, since broken lines in $\mathfrak{B}^{(k)}_q(P)$ only break at walls of $\mathscr{S}_k$ and $\mathscr{S}_k$ contains only finitely many walls. So $\mathfrak{B}_{p,q}(P)$ is finite as well.
\end{proof}

\section{Tropical disks and curves}								

We use the definition of tropical curves from \cite{Gra}, Definition 3.3. Tropical curves may have bounded legs ending in affine singularities. They do not have uni- or bivalent vertices and fulfill the ordinary balancing condition $\sum_{E\ni V} u_{(V,E)}=0$ at vertices $V$, if not specified otherwise. Here $u_{(V,E)}\in\iota_\star\Lambda_{h(V)}$ is the weight vector of the flag $(V,E)$, as defined in \cite{Gra}, Definition 3.3.

\begin{defi}
A \textit{tropical disk} $h^\circ : \Gamma \rightarrow B$ is a tropical curve with a unique univalent vertex $V_\infty$, such that $h^\circ$ is balanced for all vertices $V \neq V_\infty$. 

Write $\mathfrak{H}_q^\circ(P)$ for the set of all rational tropical disks on $B$ with one unbounded leg of weight $q$ and $h^\circ(V_\infty)=P$. Let $\mathfrak{H}_{p,q}^\circ(P)$ be the set of tropical disks in $\mathfrak{H}_q^\circ(P)$ with $u_{(V_\infty,E_\infty)}=-pm_{\textup{out}}$, where $E_\infty$ is the unique edge adjacent to $V_\infty$.
\end{defi}

\begin{prop}[\cite{CPS}, Lemma 6.4]
\label{prop:bhcirc}
There is a surjective map
\[ \mu : \mathfrak{H}^\circ_q(P) \rightarrow \mathfrak{B}_q(P) \]
such that for each $\mathfrak{b}\in\mathfrak{B}_q(P)$ the preimage $\mu^{-1}(\mathfrak{b})$ is finite. We say a tropical disk in $\mu^{-1}(\mathfrak{b})$ is obtained from $\mathfrak{b}$ by disk completion.
\end{prop}

\begin{proof}
Let $h^\circ : \Gamma \rightarrow B$ be a tropical disk in $\mathfrak{H}^\circ_q(P)$. Since $\Gamma$ is a tree, there is a unique path from the unbounded edge to the vertex $V_\infty$. We show that this path defines a broken line $\mathfrak{b}=\mu(h)$ in $\mathfrak{B}_q(P)$. The complement of the path is a disjoint union of tropical disks without unbounded legs (Maslov index zero). Such tropical disks map onto the $1$-skeleton of $\mathscr{S}_\infty$ (see \cite{Gra}, Proposition 5.20, and the proof of \cite{CPS}, Lemma 6.4). The endpoint of these disks map onto walls of $\mathscr{S}_\infty$ and give the break points of $\mathfrak{b}$. The coefficients of the monomials attached to $\mathfrak{b}$ are determined by $a_1=1$ and the breaking calculations (Definition \ref{defi:broken}, (2)). 

To show that the map $\mu$ is surjective and has finite preimages, we describe the set of preimages of a broken line $\mathfrak{b}\in\mathfrak{B}_q(P)$. Whenever $\mathfrak{b}$ breaks at a wall $\mathfrak{p}$, add bounded edges/legs that map onto $\mathfrak{p}$ and stop at the break point. We know the exponent that $\mathfrak{b}$ picks up by breaking. It is a multiple of the exponent of $\mathfrak{p}$. The edges/legs we add must have total weight equal to this multiple in order to produce a balanced tropical disk. There are several possibilities, corresponding to partitions of this multiple. Repeat this procedure for all ancestors of $\mathfrak{p}$. By ancestors we mean all walls in $\mathscr{S}_\infty$ that are involved in the scattering to produce $\mathfrak{p}$. In each step we have several possibilities, corresponding to partitions of the nontrivial monomials involved. All (finitely many) choices of such partitions give a preimage of $\mathfrak{b}$ and all preimages are given this way.
\end{proof}

Let $P$ be a general point in an unbounded chamber of $\mathscr{S}_{p+q}$. Since $\mu$ maps broken lines in $\mathfrak{B}_{p,q}(P)$ to tropical disks in $\mathfrak{H}_{p,q}^\circ(P)$, we have the following.

\begin{cor}
\label{cor:bhcirc}
There is a surjective map with finite preimages
\[ \mu : \mathfrak{H}_{p,q}^\circ(P) \rightarrow \mathfrak{B}_{p,q}(P). \]
\end{cor}

\begin{cor}
We have a decomposition
\[ \mathfrak{H}_q^\circ(P) = \coprod_{p=1}^\infty \mathfrak{H}_{p,q}^\circ(P). \]
In particular, for a tropical disk in $\mathfrak{H}_{p,q}^\circ(P)$ the edge $E_\infty$ is parallel to $m_{\textup{out}}$.
\end{cor}

\begin{proof}
This follows from Corollary \ref{cor:broken} and the construction of $\mu$ in Proposition \ref{prop:bhcirc} and Corollary \ref{cor:bhcirc}.
\end{proof}

\begin{defi}
Let $\mathfrak{H}_{p,q}(P)$ be the set of tropical curves on $B$ having two unbounded legs, of weight $p$ and $q$, and such that the image of the unbounded leg of weight $p$ contains $P$.
\end{defi}

\begin{prop}
\label{prop:hcirch}
There is a bijective map
\[ \mathfrak{H}^\circ_{p,q}(P) \rightarrow \mathfrak{H}_{p,q}(P), h^\circ \mapsto h \]
\end{prop}

\begin{proof}
By Proposition \ref{prop:parallel} and Corollary \ref{cor:bhcirc} tropical disks in $\mathfrak{H}^\circ_{p,q}(P)$ have ending edge $E_\infty$ parallel to $m_{\textup{out}}$. We obtain a tropical curve by completing $E_\infty$ to an unbounded leg $E_{\textup{out}}$.
\end{proof}

\begin{cor}
\label{cor:bh}
There is a surjective map with finite preimages
\[ \mu : \mathfrak{H}_{p,q}(P) \rightarrow \mathfrak{B}_{p,q}(P). \]
\end{cor}

\begin{proof}
The map is given by composing the inverse of the map from Proposition \ref{prop:hcirch} with the map $\mu$ from Proposition \ref{prop:bhcirc}.
\end{proof}

\begin{cor}
\label{cor:finite}
The set $\mathfrak{H}_{p,q}(P)$ is finite.
\end{cor}

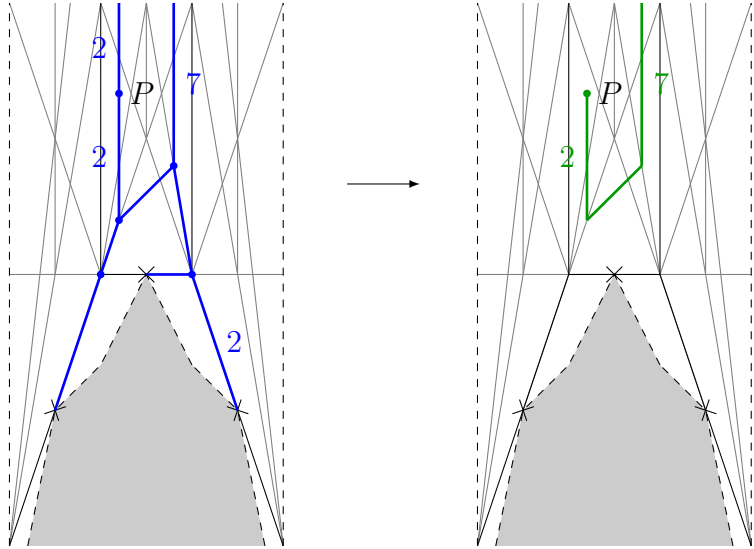
\begin{figure}[h!]
\centering
\begin{tikzpicture}[scale=1.2,rotate=90]
\draw[gray] (0,0.5) -- (0,2);
\draw[gray] (0,0.5) -- (0,-1);
\draw[gray] (-1.5,-0.5) -- (3,1);
\draw[gray] (-1.5,1.5) -- (3,0);
\draw[gray] (-3,-1) -- (3,0);
\draw[gray] (-3,2) -- (3,1);
\draw[gray] (1.5,0.5) -- (3,0.5);
\draw[gray] (0,-0.5) -- (3,-0.5);
\draw[gray] (0,1.5) -- (3,1.5);
\draw[gray] (0,0) -- (3,0.5);
\draw[gray] (0,0) -- (3,-1);
\draw[gray] (0,1) -- (3,0.5);
\draw[gray] (0,1) -- (3,2);
\draw[gray] (-3,-1) -- (3,-1/3);
\draw[gray] (-3,2) -- (3,1+1/3);
\coordinate (1) at (0,0);
\coordinate (2) at (0,1);
\coordinate (3) at (-3,-1);
\coordinate (4) at (-3,2);
\draw (3) -- (1) -- (2) -- (4);
\draw (1) -- (3,0);
\draw (2) -- (3,1);
\draw[dashed] (3) -- (3,-1);
\draw[dashed] (4) -- (3,2);
\draw[dashed] (-1,0) -- (0,0.5) -- (-1,1);
\draw[dashed] (-3,-0.8) -- (-1.5,-0.5) -- (-1,0);
\draw[dashed] (-3,1.8) -- (-1.5,1.5) -- (-1,1);
\fill[opacity=0.2] (-3,1.8) -- (-1.5,1.5) -- (-1,1) -- (0,0.5) -- (-1,0) -- (-1.5,-0.5) -- (-3,-0.8);
\coordinate[fill,cross,inner sep=3pt] (5) at (0,0.5);
\coordinate[fill,cross,inner sep=3pt,rotate=26.57] (6) at (-1.5,-0.5);
\coordinate[fill,cross,inner sep=3pt,rotate=-26.57] (7) at (-1.5,1.5);
\fill[color=blue] (0,0) circle (1.2pt);
\fill[color=blue] (0,1) circle (1.2pt);
\fill[color=blue] (0.6,0.8) circle (1.2pt);
\fill[color=blue] (2,0.8) circle (1.2pt);
\fill[color=blue] (1.2,0.2) circle (1.2pt);
\draw[blue,line width=1pt] (-1.5,-0.5) -- (0,0) node[midway,right]{$2$} -- (1.2,0.2) -- (0.6,0.8) -- (-1.5,1.5);
\draw[blue,line width=1pt] (0,0) -- (0,0.5);
\draw[blue,line width=1pt] (1.2,0.2) -- (3,0.2) node[midway,right]{$7$};
\draw[blue,line width=1pt] (0.6,0.8) -- (2,0.8) node[midway,left]{$2$};
\draw[blue,line width=1pt] (2,0.8) -- (3,0.8) node[midway,left]{$2$};
\draw (2,0.8) node[right]{$P$};
\draw[->] (1,-1.7) -- (1,-2.5);
\draw (1,-3);
\end{tikzpicture}
\begin{tikzpicture}[scale=1.2,rotate=90]
\draw[gray] (0,0.5) -- (0,2);
\draw[gray] (0,0.5) -- (0,-1);
\draw[gray] (-1.5,-0.5) -- (3,1);
\draw[gray] (-1.5,1.5) -- (3,0);
\draw[gray] (-3,-1) -- (3,0);
\draw[gray] (-3,2) -- (3,1);
\draw[gray] (1.5,0.5) -- (3,0.5);
\draw[gray] (0,-0.5) -- (3,-0.5);
\draw[gray] (0,1.5) -- (3,1.5);
\draw[gray] (0,0) -- (3,0.5);
\draw[gray] (0,0) -- (3,-1);
\draw[gray] (0,1) -- (3,0.5);
\draw[gray] (0,1) -- (3,2);
\draw[gray] (-3,-1) -- (3,-1/3);
\draw[gray] (-3,2) -- (3,1+1/3);
\coordinate (1) at (0,0);
\coordinate (2) at (0,1);
\coordinate (3) at (-3,-1);
\coordinate (4) at (-3,2);
\draw (3) -- (1) -- (2) -- (4);
\draw (1) -- (3,0);
\draw (2) -- (3,1);
\draw[dashed] (3) -- (3,-1);
\draw[dashed] (4) -- (3,2);
\draw[dashed] (-1,0) -- (0,0.5) -- (-1,1);
\draw[dashed] (-3,-0.8) -- (-1.5,-0.5) -- (-1,0);
\draw[dashed] (-3,1.8) -- (-1.5,1.5) -- (-1,1);
\fill[opacity=0.2] (-3,1.8) -- (-1.5,1.5) -- (-1,1) -- (0,0.5) -- (-1,0) -- (-1.5,-0.5) -- (-3,-0.8);
\coordinate[fill,cross,inner sep=3pt] (5) at (0,0.5);
\coordinate[fill,cross,inner sep=3pt,rotate=26.57] (6) at (-1.5,-0.5);
\coordinate[fill,cross,inner sep=3pt,rotate=-26.57] (7) at (-1.5,1.5);
\fill[color=black!40!green] (2,0.8) circle (1.2pt);
\draw[black!40!green,line width=1pt] (1.2,0.2) -- (3,0.2) node[midway,right]{$7$};
\draw[black!40!green,line width=1pt] (1.2,0.2) -- (0.6,0.8);
\draw[black!40!green,line width=1pt] (0.6,0.8) -- (2,0.8) node[midway,left]{$2$};
\draw (2,0.8) node[right]{$P$};
\end{tikzpicture}
\caption{An example of the map $\mu : \mathfrak{H}_{2,7}(P) \rightarrow \mathfrak{B}_{2,7}(P)$. Splitting the leg with weight $2$ on the left into two legs with weight $1$ each gives a tropical curve with the same image.}
\label{fig:mu}
\end{figure}

\begin{defi}
\label{defi:mult}
Let $h : \Gamma \rightarrow B$ be a tropical curve. Write $V(\Gamma)$ for the set of vertices of $\Gamma$ and $L_\Delta(\Gamma)$ for the set of bounded legs, which necessarily have to end in affine singularities of $B$. For a trivalent vertex $V\in V(\Gamma)$ define, with $u_{(V,E)}\in\iota_\star\Lambda_{h(V)}$ as in \cite{Gra}, Definition 3.3,
\[ m_V=\lvert u_{(V,E_1)}\wedge u_{(V,E_2)}\rvert=\lvert\text{det}(u_{(V,E_1)}|u_{(V,E_2)})\rvert, \]
where $E_1,E_2$ are any two edges adjacent to $V$. For a vertex $V\in V(\Gamma)$ of valency $\nu_V>3$ let $h_V$ be the one-vertex tropical curve describing $h$ locally at $V$ and let $h'_V$ be a deformation of $h_V$ to a trivalent tropical curve. 
This deformation has $\nu_V-2$ vertices. We define 
\[ m_V=\prod_{V'\in V(h'_V)}m_{V'} \]
and, by Proposition 2.7 in \cite{GPS}, this expression is independent of the deformation $h'_V$, hence well-defined. For a bounded leg $E\in L_\Delta(\Gamma)$ with weight $w_E$ define 
\[ m_E=\frac{(-1)^{w_E+1}}{w_E^2}. \]
We define the \textit{multiplicity} of $h$ to be
\[ m_h = \frac{1}{|\text{Aut}(h)|} \cdot \prod_V m_V \cdot \prod_{E\in L_\Delta(\Gamma)} m_E. \]
\end{defi}

The \textit{class} $\beta$ of a tropical curve $h$ is the class of the corresponding stable log map. For tropical curves not mapping to the bounded maximal cell $\sigma_0$ it can be read off from the tropical curve via projection to the unbounded direction or, equivalently, via deformation of the toric degeneration, as in \cite{Gra}, {\S}3.4. But tropical curves in $\mathfrak{H}_{p,q}$ may have image intersecting $\sigma_0$, so we need another way to read off $\beta$ from $h$.

In the dual intersection complex $B$ of the toric degeneration $\mathfrak{X}\rightarrow\mathbb{A}^1$ there are certain \textit{elementary tropical curves}. They have one unbounded leg with image given by an unbounded edge of $B$ and two bounded legs with weight $1$. Projection to the unbounded direction or deformation of the toric degeneration as in \cite{Gra}, {\S}3.4, easily shows that they correspond to the toric divisors of $X_0$. Now given any tropical curve $h : \Gamma \rightarrow B$ we can read of the class of $h$ from its intersection with the elementary tropical curves. If the intersection is not transversal we can use the moving lemma of tropical intersection theory and consider a small translation of one of the curves. For tropical curves $\tilde{h} : \tilde{\Gamma} \rightarrow \tilde{B}$ the class can be read off from any of its preimages under the map from Lemma \ref{lem:tilde}. This is well-defined, since different preimages correspond to different edge splittings without changing the total weight, and this does not affect the intersection multiplicity.

\begin{defi}
\label{defi:class}
Let $\mathfrak{H}_{p,q}^\beta(P)$ be the set of tropical curves in $\mathfrak{H}_{p,q}(P)$ of class $P$. Let $\mathfrak{B}_{p,q}^\beta(P)$ be the image of $\mathfrak{H}_{p,q}^\beta(P)$ under $\mu$.
\end{defi}

\begin{expl}
Consider the tropical curve from Figure \ref{fig:tilde}. It has nontrivial intersection with two elementary tropical curves, both of them corresponding to a line $L$ in $\mathbb{P}^2$. The intersection multiplicities are $1$ and $2$, respectively. Hence, the tropical curve corresponds to curve class $L+2L=3L$ on $\mathbb{P}^2$.
\end{expl}

\begin{figure}[h!]
\centering
\vspace{-3mm}
\begin{tikzpicture}[scale=0.8,rotate=90]
\fill[color=blue] (0,0) circle (1.2pt);
\fill[color=blue] (0,1) circle (1.2pt);
\fill[color=blue] (0.6,0.8) circle (1.2pt);
\fill[color=blue] (2,0.8) circle (1.2pt);
\fill[color=blue] (1.2,0.2) circle (1.2pt);
\draw[blue,line width=1pt] (-1.5,-0.5) -- (0,0) node[midway,right]{$2$} -- (1.2,0.2) -- (0.6,0.8) -- (-1.5,1.5);
\draw[blue,line width=1pt] (0,0) -- (0,0.5);
\draw[blue,line width=1pt] (1.2,0.2) -- (3,0.2) node[midway,right]{$7$};
\draw[blue,line width=1pt] (0.6,0.8) -- (2,0.8) node[midway,left]{$2$};
\draw[blue,line width=1pt] (2,0.8) -- (3,0.8) node[midway,left]{$2$};
\draw (2,0.8) node[right]{$P$};
\draw (-0.2,-0.2) -- (-0.2,0.3);
\draw (-0.2,-0.2) -- (-1.7,-0.7);
\draw (-0.2,-0.2) -- (2.8,-0.2) node[above]{$L$};
\draw (-0.2,1.2) -- (-0.2,0.7);
\draw (-0.2,1.2) -- (-1.7,1.7);
\draw (-0.2,1.2) -- (2.8,1.2) node[above]{$L$};
\draw[<-] (-0.2,-0.0667) to[bend right=10] (-0.7,0.1) node[below]{$2$};
\draw[<-] (-0.2,1.0667) to[bend left=10] (-0.7,0.9) node[below]{$1$};
\end{tikzpicture}
\caption{The tropical curve from Figure \ref{fig:tilde} has class $3L$.}
\label{fig:int}
\end{figure}
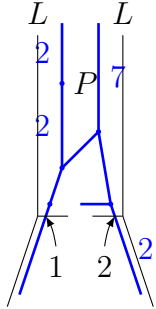

\section{Tropical correspondence}							

In \cite{Gra} the author established a correspondence between wall functions and $1$-marked log Gromov-Witten invariants. Here we prove a similar correspondence between broken lines and $2$-marked log Gromov-Witten invariants. The steps of the proof are the same, corresponding to the subsections of this section:
\begin{compactenum}
\item resolve the log singularities to obtain a log smooth degeneration; 
\item refine $\mathscr{P}$ by all walls and also by the broken line (that's new here) to obtain toric transversality;
\item apply the degeneration formula and show that for each tropical curve there is only one possible way of gluing; 
\item use the correspondence between walls and tropical disks without unbounded legs (Maslov index zero) to show that broken lines count tropical disks with one unbounded leg (Maslov index two).
\end{compactenum}
Steps (1)-(3) will give the following tropical correspondence theorem:
\[ N_{p,q} = p \cdot \sum_{h\in\mathfrak{H}_{p,q}(P)}\textup{Mult}(h). \]
Step (4) will show that coefficients of broken lines are counts of disk completions:
\[ a_{\mathfrak{b}} = \sum_{h\in\mu^{-1}(\mathfrak{b})} \textup{Mult}(h). \]
Together with Corollary \ref{cor:bh} this gives the following formula
\[ N_{p,q}=p\cdot\sum_{\mathfrak{b}\in\mathfrak{B}_{p,q}(P)}a_{\mathfrak{b}}. \]
Using the definition of theta functions via broken lines this will give a proof of Theorems \ref{thm:main} and \ref{thm:main2}, as we will explain in detail in {\S}\ref{S:theta}.

\subsection{Resolution of singularities}							
\label{S:resolution}

To apply the degeneration formula we need a log smooth family. Unfortunately, the log structure of our toric degeneration is not even coherent (there is no chart for the log structure at the points corresponding to the affine singularities). However, we can obtain a log smooth degeneration without changing the general fiber by successively blowing up $\mathfrak{X}$ along the irreducible components of the central fiber. Depending on the number of components there are several ways to do this. The easiest is to choose a cyclic ordering (with respect to the intersection combinatorics) and to blow up in this order along all but one irreducible component. Figure \ref{fig:blowup} shows the resulting central fiber (intersection complex) for $\mathbb{P}^2$. The resulting family $\tilde{\mathfrak{X}}\rightarrow\mathbb{A}^1$ has general fiber $X$ and is log smooth by \cite{GSdataII}, Lemma 2.12. In \cite{Gra}, {\S}2, the author used a more symmetric resolution. However, this resolution is not projective, so we don't use it here.

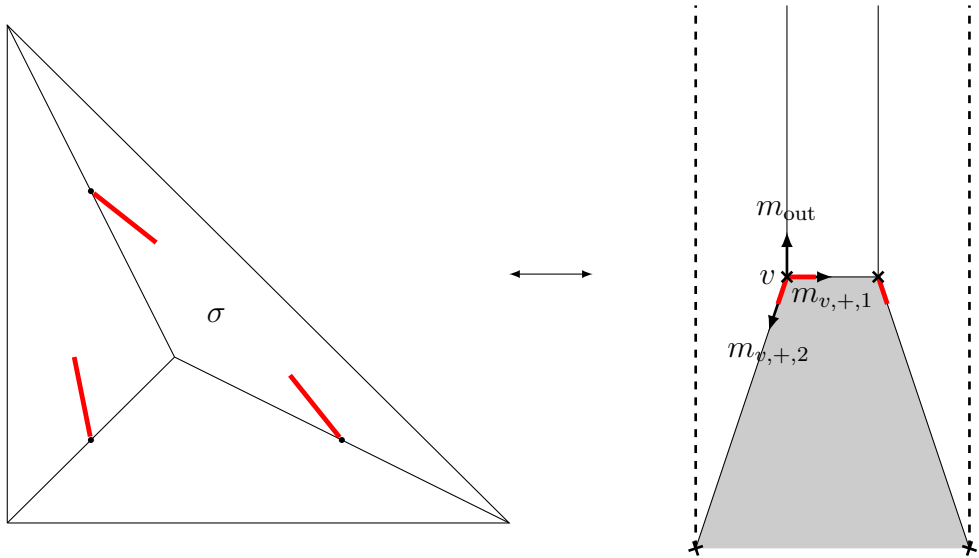
\begin{figure}[h!]
\centering
\begin{tikzpicture}[scale=2.2]
\coordinate (0) at (0,0);
\coordinate (1) at (-1,-1);
\coordinate (2) at (2,-1);
\coordinate (3) at (-1,2);
\coordinate[fill,circle,inner sep=0.8pt] (1a) at (-0.5,-0.5);
\coordinate[fill,circle,inner sep=0.8pt] (2a) at (1,-0.5);
\coordinate[fill,circle,inner sep=0.8pt] (3a) at (-0.5,1);
\draw (1) -- (2) -- (3) -- (1);
\draw (0) -- (1);
\draw (0) -- (2);
\draw (0) -- (3);
\draw[red,line width=1.8pt] (1a) -- (-0.6,0);
\draw[red,line width=1.8pt] (2a) -- (0.69,-0.11);
\draw[red,line width=1.8pt] (3a) -- (-0.11,0.69);
\draw (0.25,0.25) node{$\sigma$};
\draw[<->] (2,0.5) -- (2.5,0.5);
\draw (3,-1.2);
\end{tikzpicture}
\begin{tikzpicture}[scale=1.2,rotate=90]
\coordinate (1) at (0,0);
\coordinate[label=left:${v}$] (2) at (0,1);
\coordinate (3) at (-3,-1);
\coordinate (4) at (-3,2);
\fill[opacity=0.2] (-3,-1) -- (0,0) -- (0,1) -- (-3,2);
\draw (3) -- (1) -- (2) -- (4);
\draw (1) -- (3,0);
\draw (2) -- (3,1);
\draw[line width=1pt,dashed] (3) -- (3,-1);
\draw[line width=1pt,dashed] (4) -- (3,2);
\draw[line width=1pt,->] (0,1) -- (-0.6,1.2) node[below]{$m_{v,+,2}$};
\draw[line width=1pt,->] (0,1) -- (0,0.5) node[below]{$m_{v,+,1}$};
\draw[line width=1pt,->] (0,1) -- (0.5,1) node[above]{$m_{\textup{out}}$};
\draw[red,line width=1.8pt] (0,1) -- (-0.3,1.1);
\draw[red,line width=1.8pt] (0,1) -- (0,0.684);
\draw[red,line width=1.8pt] (0,0) -- (-0.3,-0.1);
\coordinate[fill,cross,line width=1pt,inner sep=2pt] (5) at (0,0);
\coordinate[fill,cross,line width=1pt,inner sep=2pt] (6) at (0,1);
\coordinate[fill,cross,line width=1pt,inner sep=2pt,rotate=26.57] (7) at (-3,-1);
\coordinate[fill,cross,line width=1pt,inner sep=2pt,rotate=-26.57] (8) at (-3,2);
\end{tikzpicture}
\caption{The intersection complex and its dual $(\tilde{B},\mathscr{P},\varphi)$ for a resolution of the toric degeneration of $(\mathbb{P}^2,E)$.}
\label{fig:blowup}
\end{figure}

For later convenience we indicate the choices of small resolutions by red stubs attached to the vertices of $\mathscr{P}$. The stubs at a vertex $v$ point in the directions corresponding to the toric divisors of $X_v$ intersecting an exceptional line. We denote the primitive vectors in the direction of the red stubs adjacent to $v$ by $m_{v,+,i}\in\Lambda_{\tilde{B},v}$ for $i=1,\ldots,n$, where $n\in\{0,1,2\}$ is the number of exceptional lines contained in $X_v$. Denote the primitive vectors in the direction of the other edges of $\sigma_0$ adjacent to $v$ by $m_{v,-,i}\in\Lambda_{\tilde{B},v}$ for $i=1,\ldots,n$, where $m=2-n$. Further, $m_{\textup{out}}$ is the unique unbounded direction of $\tilde{B}$.

\begin{expl}
Figure \ref{fig:blowup} shows the intersection complex and its dual for a resolution of the toric degeneration of $(\mathbb{P}^2,E)$. One component (corresponding to $\sigma$ and $v$, respectively) contains two exceptional lines.
\end{expl}

\begin{defi}
For an effective curve class $\beta$ of $X$ and $p,q\in\mathbb{N}$ with $p+q=D\cdot\beta$ define a class of stable log maps $\beta_{p,q}$ to $\tilde{\mathfrak{X}}\rightarrow\mathbb{A}^1$ as follows:
\begin{compactenum}[(1)]
\item genus $g=0$;
\item fibers have curve class $\beta$;
\item $2$ marked points $x_p,x_q$ having contact orders $p,q$ with $\mathfrak{D}$.
\end{compactenum}
\end{defi}

Let $\mathscr{M}(\tilde{\mathfrak{X}}/\mathbb{A}^1,\beta_{p,q})$ be the moduli space of stable log maps to $\tilde{\mathfrak{X}}\rightarrow\mathbb{A}^1$ of class $\beta_{p,q}$. By \cite{GSloggw} this is a proper Deligne-Mumford stack and admits a virtual fundamental class $\llbracket\mathscr{M}(\mathfrak{X}/\mathbb{A}^1,\beta_{p,q})\rrbracket$. It has virtual dimension $1$, since the contact orders cut down the virtual dimension by $(p-1)+(q-1)=D\cdot\beta-2$. Let $\textup{ev} : \mathscr{M}(\tilde{\mathfrak{X}}/\mathbb{A}^1,\beta_{p,q}) \rightarrow \mathfrak{D}$ be the evaluation map at $x_p$.

\begin{defi}
\label{defi:N}
Define the $2$-marked log Gromov-Witten invariant
\[ N_{p,q}(X,\beta) = \int_{\llbracket\mathscr{M}(\mathfrak{X}/\mathbb{A}^1,\beta_{p,q})\rrbracket} \textup{ev}^\star[\textup{pt}] \in \mathbb{Q}. \]
Since log Gromov-Witten invariants are constant in log smooth families (\cite{MR}, Appendix A), this agrees with the definition of $N_{p,q}(X,\beta)$ in the introduction.
\end{defi}

\begin{rem}
We believe that there should be a definition of log Gromov-Witten invariants for relatively coherent targets that captures all information about the resolution described here. This would simplify the subsequent proof a lot as we wouldn't have to translate between the toric degeneration and its resolution.
\end{rem}

\subsection{Tropical curves and refinement}						

It turns out that tropical curves on $\tilde{B}$ that are tropicalizations of stable log maps to $\tilde{X}$ do not fulfill the ordinary balancing condition, but the following modified one.

\begin{defi}
\label{defi:types}
For $p,q\in\mathbb{Z}_{>0}$ and $P$ a general point in a unbounded chamber of $\mathscr{S}_{p+q}$ let $\tilde{\mathfrak{H}}_{p,q}(P)$ be the set of tropical curves on $\tilde{B}$ with two unbounded legs, of weights $p$ and $q$, with vertex of the unbounded leg of weight $p$ being bivalent and mapping to $P$, and such that each vertex is of one of the following types:
\begin{compactenum}[(I)]
\item $V$ is not mapped to a vertex of $\mathscr{P}$. Then the ordinary balancing condition holds:
\vspace{-3mm}
\[ \sum_{E\ni V}u_{(V,E)} = 0. \]
The sum is over all edges or legs $E\in E(\Gamma_C)\cup L(\Gamma_C)$ containing $V$.
\item $V$ is mapped to a vertex $v$ of $\mathscr{P}$, and is $1$-valent with adjacent edge $E$ mapped onto the edge of $\mathscr{P}$ containing the red stub adjacent to $v$. Then the balancing condition reads, with $m_{v,+,i}$ as in Figure \ref{fig:blowup} for some $i$,
\[ u_{(V,E)} = km_{v,+,i}. \]
\item $V$ is mapped to a vertex $v$ of $\mathscr{P}$ and has exactly one adjacent edge or leg $E_{V,\textup{out}}$ that is not mapped onto a compact edge of $\mathscr{P}$. All other edges (possibly none) are compact with other vertex of type (II) above. In this case, for some $k_i\geq 0$, the following balancing condition holds:
\[ \sum_{E\ni V} u_{(V,E)} + \sum_{i=1}^{n(v)}k_im_{v,+,i} = 0. \]
\end{compactenum}
Write $V_{(I)}(\tilde{\Gamma})$, $V_{(II)}(\tilde{\Gamma})$, $V_{(III)}(\tilde{\Gamma})$ for the set of vertices of type (I), (II), (III).
\end{defi}

\begin{lem}
\label{lem:tilde}
There is a surjective map $\mathfrak{H}_{p,q}(P) \rightarrow \tilde{\mathfrak{H}}_{p,q}(P)$ by deleting bounded legs in the directions $m_{v,+,i}$ (the directions of the red stubs) and extending bounded legs in the directions $m_{v,-,i}$. See Figure \ref{fig:tilde} for an example and \cite{Gra}, Construction 3.17, for details of the construction.
\end{lem}

\begin{cor}
The set $\tilde{\mathfrak{H}}_{p,q}(P)$ is finite.
\end{cor}

\begin{proof}
This follows from Corollary \ref{cor:finite} and Lemma \ref{lem:tilde}.
\end{proof}

\begin{figure}[h!]
\centering
\begin{tikzpicture}[scale=1.2,rotate=90]
\draw[gray] (0,0.5) -- (0,2);
\draw[gray] (0,0.5) -- (0,-1);
\draw[gray] (-1.5,-0.5) -- (3,1);
\draw[gray] (-1.5,1.5) -- (3,0);
\draw[gray] (-3,-1) -- (3,0);
\draw[gray] (-3,2) -- (3,1);
\draw[gray] (1.5,0.5) -- (3,0.5);
\draw[gray] (0,-0.5) -- (3,-0.5);
\draw[gray] (0,1.5) -- (3,1.5);
\draw[gray] (0,0) -- (3,0.5);
\draw[gray] (0,0) -- (3,-1);
\draw[gray] (0,1) -- (3,0.5);
\draw[gray] (0,1) -- (3,2);
\draw[gray] (-3,-1) -- (3,-1/3);
\draw[gray] (-3,2) -- (3,1+1/3);
\coordinate (1) at (0,0);
\coordinate (2) at (0,1);
\coordinate (3) at (-3,-1);
\coordinate (4) at (-3,2);
\draw (3) -- (1) -- (2) -- (4);
\draw (1) -- (3,0);
\draw (2) -- (3,1);
\draw[dashed] (3) -- (3,-1);
\draw[dashed] (4) -- (3,2);
\draw[dashed] (-1,0) -- (0,0.5) -- (-1,1);
\draw[dashed] (-3,-0.8) -- (-1.5,-0.5) -- (-1,0);
\draw[dashed] (-3,1.8) -- (-1.5,1.5) -- (-1,1);
\fill[opacity=0.2] (-3,1.8) -- (-1.5,1.5) -- (-1,1) -- (0,0.5) -- (-1,0) -- (-1.5,-0.5) -- (-3,-0.8);
\coordinate[fill,cross,inner sep=3pt] (5) at (0,0.5);
\coordinate[fill,cross,inner sep=3pt,rotate=26.57] (6) at (-1.5,-0.5);
\coordinate[fill,cross,inner sep=3pt,rotate=-26.57] (7) at (-1.5,1.5);
\fill[color=blue] (0,0) circle (1.2pt);
\fill[color=blue] (0,1) circle (1.2pt);
\fill[color=blue] (0.6,0.8) circle (1.2pt);
\fill[color=blue] (2,0.8) circle (1.2pt);
\fill[color=blue] (1.2,0.2) circle (1.2pt);
\draw[blue,line width=1pt] (-1.5,-0.5) -- (0,0) node[midway,right]{$2$} -- (1.2,0.2) -- (0.6,0.8) -- (-1.5,1.5);
\draw[blue,line width=1pt] (0,0) -- (0,0.5);
\draw[blue,line width=1pt] (1.2,0.2) -- (3,0.2) node[midway,right]{$7$};
\draw[blue,line width=1pt] (0.6,0.8) -- (2,0.8) node[midway,left]{$2$};
\draw[blue,line width=1pt] (2,0.8) -- (3,0.8) node[midway,left]{$2$};
\draw (2,0.8) node[right]{$P$};
\draw[->] (1,-1.7) -- (1,-2.5);
\draw (1,-3);
\end{tikzpicture}
\begin{tikzpicture}[scale=1.2,rotate=90]
\draw[gray] (0,0.5) -- (0,2);
\draw[gray] (0,0.5) -- (0,-1);
\draw[gray] (-1.5,-0.5) -- (3,1);
\draw[gray] (-1.5,1.5) -- (3,0);
\draw[gray] (-3,-1) -- (3,0);
\draw[gray] (-3,2) -- (3,1);
\draw[gray] (1.5,0.5) -- (3,0.5);
\draw[gray] (0,-0.5) -- (3,-0.5);
\draw[gray] (0,1.5) -- (3,1.5);
\draw[gray] (0,0) -- (3,0.5);
\draw[gray] (0,0) -- (3,-1);
\draw[gray] (0,1) -- (3,0.5);
\draw[gray] (0,1) -- (3,2);
\draw[gray] (-3,-1) -- (3,-1/3);
\draw[gray] (-3,2) -- (3,1+1/3);
\coordinate (1) at (0,0);
\coordinate (2) at (0,1);
\coordinate (3) at (-3,-1);
\coordinate (4) at (-3,2);
\fill[opacity=0.2] (-3,-1) -- (0,0) -- (0,1) -- (-3,2);
\draw (3) -- (1) -- (2) -- (4);
\draw (1) -- (3,0);
\draw (2) -- (3,1);
\draw[dashed] (3) -- (3,-1);
\draw[dashed] (4) -- (3,2);
\draw[red,line width=1.8pt] (0,1) -- (-0.3,1.1);
\draw[red,line width=1.8pt] (0,1) -- (0,0.684);
\draw[red,line width=1.8pt] (0,0) -- (-0.3,-0.1);
\draw[blue,line width=1pt] (0,1.03) node[fill,circle,inner sep=1.2pt,label=left:(III)]{} -- (0.6,0.8) node[fill,circle,inner sep=1.2pt,label=left:(I)]{} -- (1.2,0.2) node[fill,circle,inner sep=1.2pt,label=right:(I)]{} -- (0,-0) node[fill,circle,inner sep=1.2pt,label=right:(III)]{};
\draw[blue,line width=1pt] (1.2,0.2) -- (3,0.2) node[midway,right]{$7$};
\draw[blue,line width=1pt] (0.6,0.8) -- (2,0.8) node[midway,left]{$2$} -- (3,0.8) node[midway,left]{$2$};
\draw[blue,line width=1pt] (2,0.8) node[left]{(I)};
\fill[blue] (2,0.8) circle (1.2pt);
\draw[blue,line width=1pt] (0,0) -- (0,0.97) node[fill,circle,inner sep=1.2pt,label=below:(II)]{};
\draw (2,0.8) node[right]{$P$};
\coordinate[fill,cross,inner sep=3pt] (5) at (0,0);
\coordinate[fill,cross,inner sep=3pt] (6) at (0,1);
\coordinate[fill,cross,inner sep=3pt,rotate=26.57] (7) at (-3,-1);
\coordinate[fill,cross,inner sep=3pt,rotate=-26.57] (8) at (-3,2);
\end{tikzpicture}
\caption{An example of the map $\mathfrak{H}_{2,7} \rightarrow \tilde{\mathfrak{H}}_{2,7}$. The weights of edges and the types (I)-(III) of vertices of the target are given. Two vertices are mapped to the same vertex, but not connected.}
\label{fig:tilde}
\end{figure}
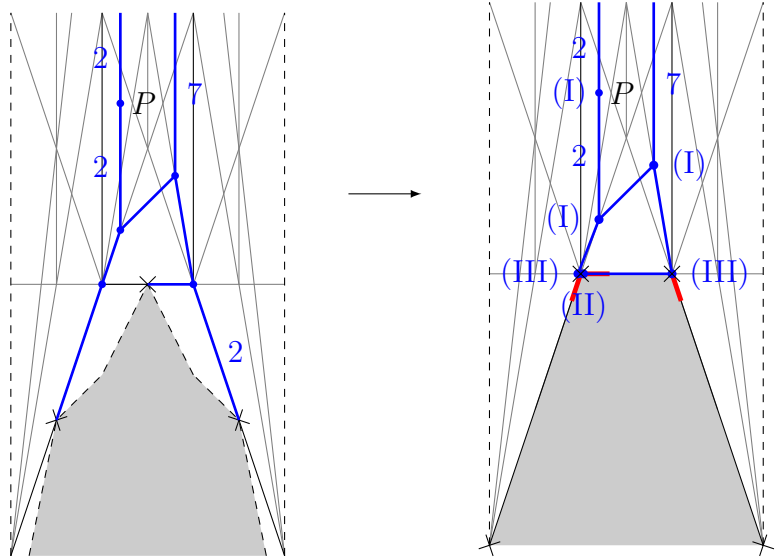

\begin{defi}
For $p,q\in\mathbb{Z}_{>0}$, a general point $P$ in an unbounded chamber of $\mathscr{S}_{p+q}$ and an effective curve class $\beta$ on $X$, let $\tilde{\mathfrak{H}}_{p,q}^\beta(P)$ be the moduli space of tropical curves in $\tilde{\mathfrak{H}}_{p,q}(P)$ whose image under the map from Lemma \ref{lem:tilde} has class $\beta$ (see Definition \ref{defi:class}. Note that
\[ \tilde{\mathfrak{H}}_{p,q}(P) = \coprod_\beta \tilde{\mathfrak{H}}_{p,q}^\beta(P), \]
where the sum is over all effective curve classes on $X$ with $D\cdot\beta=p+q$.
\end{defi}

\begin{prop}
\label{prop:balancing}
$\tilde{\mathfrak{H}}_{p,q}^\beta(P)$ is the set of tropical curves that arise as tropicalizations of stable log maps in $\mathscr{M}(\tilde{X},\beta_{p,q})$. Vertices of type (II) correspond to components that are multiple covers of an exceptional $\mathbb{P}^1$, and vertices of type (III) to components intersecting an exceptional $\mathbb{P}^1$.
\end{prop}

\begin{proof}
It is shown in \cite{Gra}, Proposition 3.12, that tropicalizations of stable log maps to $\tilde{\mathfrak{X}}$ have vertices of types (I)-(III) above. The condition to have two marked points with contact orders $p$ and $q$ is equivalent to the condition on the tropical curve to have two unbounded legs of weight $p$ and $q$, respectively.
\end{proof}

\begin{rem}[Tropical cycles]
More formally, the ``elementary tropical curves'' should be seen as tropical cocycles in $H^1(\overline{B},\iota_\star\check{\Lambda}_N)$ and the ``intersection'' is given by an extension of the pairing $H_1(B,\iota_\star\Lambda)\otimes H^1(B,\iota_\star\check{\Lambda})\rightarrow\mathbb{Z}^2$ from \cite{Rud}. Here $\overline{B}$ is a compactification of $B$ and $\check{\Lambda}_N$ is an extension of $\Lambda$ such that the stalk of $\Lambda_N$ at a point in $\overline{B}\setminus B$ is generated by the $m_{\text{out}}$. It follows from \cite{RZ} there is a natural isomorphism $H^1(\overline{B},\iota_\star\check{\Lambda}_N)\cong\text{Pic}(X)$. This gives the toric divisor (line bundle) corresponding to the ``elementary tropical curve''. For a more detailed discussion see {\S}3 of \cite{GRZ}.
\end{rem}

\begin{con}[Refinement]
For $p,q\in\mathbb{Z}_{>0}$ let $P$ be a general point in an unbounded chamber of $\mathscr{S}_{p+q}$. Let $\mathscr{P}_{p,q}$ be a refinement of $\mathscr{P}$ such that all tropical curves in $\mathfrak{H}_{p,q}(P)$ are contained in the $1$-skeleton of $\mathscr{P}_{p,q}$ and $P$ is a vertex of $\mathscr{P}_{p,q}$. This induces a logarithmic modification $\tilde{\mathfrak{X}}_{p,q}\rightarrow\mathbb{A}^1$ of $\tilde{\mathfrak{X}}\rightarrow\mathbb{A}^1$ via subdivision of Artin fans, see \cite{AW}. By making a base change $t\mapsto t^e$ we can scale $\mathscr{P}_{p,q}$ and thus assume it has integral vertices. By construction, all stable log maps to the central fiber $Y$ of $\tilde{\mathfrak{X}}_{p,q}\rightarrow\mathbb{A}^1$ are torically transverse.
\end{con}

\subsection{The degeneration formula}							

Gromov-Witten invariants are invariant under logarithmic modifications \cite{AW}. Hence we can compute $N_{p,q}(X,\beta)$ on $Y$, the central fiber of the degeneration $\tilde{\mathfrak{X}}_{p,q}\rightarrow\mathbb{A}^1$ constructed above,
\[ N_{p,q}(X,\beta) = \int_{\llbracket\mathscr{M}(Y,\beta_{p,q})\rrbracket}\textup{ev}^\star[\textup{pt}]. \]
Here $\textup{ev}$ is the evaluation map at $x_p$, the marked point of order $p$. 

On the central fiber $Y$ we have some techniques for computing the invariants $N_{p,q}(X,\beta)$. By the decomposition formula (Proposition \ref{prop:dec}) the connected components of the moduli space $\mathscr{M}(Y,\beta_{p,q})$ are labelled by certain (decorated) tropical curves. The gluing formula (Proposition \ref{prop:gluing}) relates the contributions of every tropical curve to contributions of its vertices and edges. In our case, similar to \cite{Gra}, the situation is particularly easy. Tropical curves in $\tilde{\mathfrak{H}}_{p,q}(P)$ have a natural orientation of edges and the gluing according to this orientation is the only one giving a nonzero contribution (Proposition \ref{prop:unique}).

\begin{prop}[Decomposition formula]
\label{prop:dec}
For $\tilde{h}\in\tilde{\mathfrak{H}}_{p,q}^\beta(P)$ let $\mathscr{M}_{\tilde{h}}$ be the moduli space of stable log maps in $\mathscr{M}(Y,\beta_{p,q})$ with tropicalization $\tilde{h}$. Then
\[ \llbracket\mathscr{M}(Y,\beta_{p,q})\rrbracket = \sum_{\tilde{h}\in\tilde{\mathfrak{H}}_{p,q}^\beta(P)}\frac{l_{\tilde{\Gamma}}}{|\textup{Aut}(\tilde{h})|} F_\star\llbracket\mathscr{M}_{\tilde{h}}\rrbracket, \]
where $l_{\tilde{\Gamma}} := \textup{lcm}\{w_E \ | \ E\in E(\tilde{\Gamma})\}$ and $F:\mathscr{M}_{\tilde{h}}\rightarrow\mathscr{M}_\beta$ is the forgetful map.
\end{prop}

\begin{proof}
This is a special case of the decomposition formula of \cite{ACGS1}, similar to \cite{Gra}, Proposition 4.4. In \cite{ACGS1} the sum is over tropical curves with vertices decorated by curve classes to the corresponding components. In our case, as in \cite{Gra}, Proposition 4.1, these curve classes are determined by the tropical curve. Hence, we can simply sum over $\tilde{\mathfrak{H}}_{p,q}(P)$. The nominator $l_{\tilde{\Gamma}}$ is the smallest integer such that scaling $\tilde{B}$ by $l_{\tilde{\Gamma}}$ leads to a tropical curve with integral vertices and edge lengths. By construction $\mathscr{P}_{p,q}$ has integral vertices and tropical curves in $\tilde{\mathfrak{H}}_{p,q}(P)$ are contained in the $1$-skeleton of $\mathscr{P}_{p,q}$ with vertices mapping to vertices of $\mathscr{P}_{p,q}$. The affine length of the image of an edge $E$ is a multiple of $w_E$. So the scaling necessary to obtain integral edge lengths is $l_{\tilde{\Gamma}} = \textup{lcm}\{w_E \ | \ E\in E(\tilde{\Gamma})\}$.
\end{proof}

Let $\tilde{h} : \tilde{\Gamma} \rightarrow \tilde{B}$ be a tropical curve in $\tilde{\mathfrak{H}}_{p,q}(P)$. For a vertex $V$ of $\tilde{\Gamma}$ define
\[ \mathscr{M}_V^\circ := \mathscr{M}(Y_{\tilde{h}(V)}^\circ,i_V^\star\beta_{p,q}), \]
where $Y_{\tilde{h}(V)}^\circ$ is the complement of the $0$-dimensional toric strata in $Y_{\tilde{h}(V)}$ and $i_V : Y_{\tilde{h}(V)}^\circ \rightarrow Y$ is the inclusion map.

For $V\in V_{II}(\tilde{\Gamma})$ (Definition \ref{defi:types}) with adjacent edge $E$, the moduli space $\mathscr{M}_V^\circ$ is proper, since it is isomorphic to the moduli space of $w_E$-fold multiple covers of $\mathbb{P}^1$ totally ramified at a point. For $V\in V(\tilde{\Gamma})\setminus V_{II}(\tilde{\Gamma})$ we obtain a proper moduli space as follows. Since tropical curves in $\tilde{\mathfrak{H}}_{p,q}(P)$ have genus $0$, the graph $\tilde{\Gamma}$ is a tree. We give it the structure of a rooted tree by choosing the vertex $V_{\textup{out}}$ of the unbounded leg of weight $q$ to be the root vertex. Then there is a natural orientation of the edges of $\tilde{\Gamma}$ by choosing edges to point from a vertex to its parent. For each vertex $V\in V(\tilde{\Gamma})\setminus V_{II}(\tilde{\Gamma})$ there is an evaluation map
\[ \textup{ev}_{V,-}^\circ : \mathscr{M}_V^\circ \rightarrow \prod_{E\rightarrow V}D_E^\circ, \]
where the product is over all edges of $\tilde{\Gamma}$ adjacent to $V$ and pointing towards $V$.

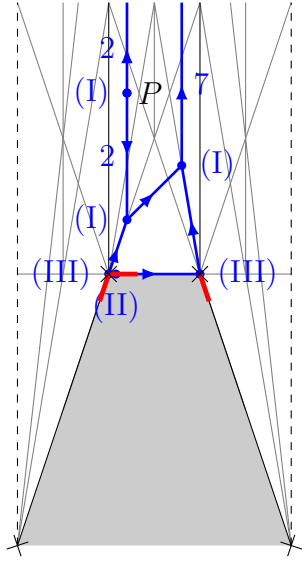
\begin{figure}[h!]
\centering
\begin{tikzpicture}[scale=1.2,rotate=90,decoration={markings,mark=at position 0.5 with {\arrow{>}}}]
\draw[gray] (0,0.5) -- (0,2);
\draw[gray] (0,0.5) -- (0,-1);
\draw[gray] (-1.5,-0.5) -- (3,1);
\draw[gray] (-1.5,1.5) -- (3,0);
\draw[gray] (-3,-1) -- (3,0);
\draw[gray] (-3,2) -- (3,1);
\draw[gray] (1.5,0.5) -- (3,0.5);
\draw[gray] (0,-0.5) -- (3,-0.5);
\draw[gray] (0,1.5) -- (3,1.5);
\draw[gray] (0,0) -- (3,0.5);
\draw[gray] (0,0) -- (3,-1);
\draw[gray] (0,1) -- (3,0.5);
\draw[gray] (0,1) -- (3,2);
\draw[gray] (-3,-1) -- (3,-1/3);
\draw[gray] (-3,2) -- (3,1+1/3);
\coordinate (1) at (0,0);
\coordinate (2) at (0,1);
\coordinate (3) at (-3,-1);
\coordinate (4) at (-3,2);
\fill[opacity=0.2] (-3,-1) -- (0,0) -- (0,1) -- (-3,2);
\draw (3) -- (1) -- (2) -- (4);
\draw (1) -- (3,0);
\draw (2) -- (3,1);
\draw[dashed] (3) -- (3,-1);
\draw[dashed] (4) -- (3,2);
\draw (2,0.8) node[right]{$P$};
\draw[blue,line width=1pt,postaction={decorate}] (2,0.8) node[fill,circle,inner sep=1.2pt,label=left:(I)]{} -- (3,0.8) node[midway,left]{$2$};
\draw[blue,line width=1pt,postaction={decorate}] (2,0.8) node[fill,circle,inner sep=1.2pt]{} -- (0.6,0.8) node[midway,left]{$2$};
\draw[blue,line width=1pt,postaction={decorate}] (0,1) node[fill,circle,inner sep=1.2pt,label=left:(III)]{} -- (0.6,0.8);
\draw[blue,line width=1pt,postaction={decorate}] (0.6,0.8) node[fill,circle,inner sep=1.2pt,label=left:(I)]{} -- (1.2,0.2);
\draw[blue,line width=1pt,postaction={decorate}] (0,0.92) node[fill,circle,inner sep=1.2pt,label=below:(II)]{} -- (0,0);
\draw[blue,line width=1pt,postaction={decorate}] (0,0) node[fill,circle,inner sep=1.2pt,label=right:(III)]{} -- (1.2,0.2);
\draw[blue,line width=1pt,postaction={decorate}] (1.2,0.2) node[fill,circle,inner sep=1.2pt,label=right:(I)]{} -- (3,0.2) node[midway,right]{$7$};
\draw[red,line width=1.8pt] (0,1) -- (-0.3,1.1);
\draw[red,line width=1.8pt] (0,1) -- (0,0.684);
\draw[red,line width=1.8pt] (0,0) -- (-0.3,-0.1);
\coordinate[fill,cross,inner sep=3pt] (5) at (0,0);
\coordinate[fill,cross,inner sep=3pt] (6) at (0,1);
\coordinate[fill,cross,inner sep=3pt,rotate=26.57] (7) at (-3,-1);
\coordinate[fill,cross,inner sep=3pt,rotate=-26.57] (8) at (-3,2);
\end{tikzpicture}
\caption{The orientation of the tropical curve from Figure \ref{fig:tilde}.}
\label{fig:orientation}
\end{figure}

\begin{lem}
\label{lem:proper}
The evaluation map $\textup{ev}_{V,-}^\circ$ is proper.
\end{lem}

\begin{proof}
Combine \cite{GPS}, Propositions 4.2 and 5.1, as in \cite{Gra}, Lemma 4.5.
\end{proof}

Since properness of morphisms is stable under base change, we obtain a proper moduli space by base change to a point $\gamma_V : \textup{Spec }\mathbb{C} \rightarrow \prod_{E\rightarrow V}D_E^\circ$, that is,
\[ \mathscr{M}_{\gamma_V} := \textup{Spec }\mathbb{C} \times_{\prod_{E\rightarrow V}D_E^\circ} \mathscr{M}_V^\circ \]
is a proper Deligne-Mumford stack.

\begin{lem}
\label{lem:vdim}
For $V\in V_{II}(\tilde{\Gamma})$ the virtual dimension of $\mathscr{M}_V$ is zero. Otherwise the virtual dimension of $\mathscr{M}_V$ equals the codimension of $\gamma_V$.
\end{lem}

\begin{proof}
See \cite{GPS}, {\S}5.3, and \cite{Gra}, Lemma 4.6.
\end{proof}

\begin{defi}
\label{defi:NV}
For a vertex $V$ of $\tilde{\Gamma}$ define
\[ N_V := \begin{cases} \int_{\llbracket\mathscr{M}_V^\circ\rrbracket}1, & V \in V_{II}(\tilde{\Gamma}); \\ \int_{\mathscr{M}_{\gamma_V}}\gamma_V^!\llbracket\mathscr{M}_V^\circ\rrbracket, & V \in V_I(\tilde{\Gamma})\cup V_{III}(\tilde{\Gamma}). \end{cases} \]
This is a finite number by Lemma \ref{lem:vdim} and independent of $\gamma_V$ by Lemma \ref{lem:proper}.
\end{defi}

\begin{prop}[\cite{Gra}, Proposition 4.8]\
\label{prop:N}
\begin{compactenum}[(I)]
\item For $V\in V_I(\tilde{\Gamma})$ let $e_1,\ldots,e_n$ be the edges of $\mathscr{P}_{p,q}$ adjacent to $\tilde{h}(V)$ and let $m_1,\ldots,m_n$ be the corresponding primitive vectors. Let $\textbf{w}_{i}=(w_{i1},\ldots,w_{il_i})$ be the weights of edges of $\tilde{\Gamma}$ mapping to $e_i$ and write $\textbf{w}=(\textbf{w}_1,\ldots,\textbf{w}_n)$. Then $N_V$ is the toric invariant $N_\textbf{m}(\textbf{w})$ as defined in \cite{GPS} and \cite{Gra}, {\S}4.1.
\item If $V\in V_{II}(\tilde{\Gamma})$, then
\[ N_V = \frac{(-1)^{w_E-1}}{w_E^2}, \]
where $E$ is the unique edge adjacent to $V$.
\item If $V\in V_{III}(\tilde{\Gamma})$, then
\[ N_V = \sum_{i=1}^{n(v)}\sum_{\textbf{w}_{V,+}} \frac{N_{\textbf{m}}(\textbf{w})}{|\textup{Aut}(\textbf{w}_{V,+})|} \prod_{i=1}^l\frac{(-1)^{w_{V,i}-1}}{w_{V,i}}. \]
Here $n(v)$ is the number of red stubs attached to $v$ (see \S\ref{S:resolution}). The second sum is over all weight vectors $\textbf{w}_{V,+,i}=(w_1,\ldots,w_{l_i})$ such that $\sum_{i=1}^{l_i} w_i = k_i$, with $k_i$ as in Proposition \ref{prop:balancing}, (III). Further, $N_{\textbf{m}}(\textbf{w})$ is a toric invariant as in (I) with $\textbf{m}=((m_{v,-,i})_i,(m_{v,+,i})_i)$ and $\textbf{w}=(((w_E)_{E\in E_{V,-,i})_i},(\textbf{w}_{V,+,i})_i)$, where $E_{V,-,i}$ is the set of edges adjacent to $V$ and mapped to direction $m_{v,-,i}$.
\end{compactenum}
The bivalent vertex $V$ mapping to $P$ is of type (I) and contributes $N_V=1$.
\end{prop}

Define $\bigtimes_{V\in V(\tilde{\Gamma})}\mathscr{M}_V$ to be the moduli space of stable log maps in $\prod_V\mathscr{M}_V$ matching over the divisors $D_E$, $E\in E(\tilde{\Gamma})$, i.e., the fiber product
\begin{equation*}
\begin{xy}
\xymatrix{
\displaystyle\bigtimes_V\mathscr{M}_V \ar[rr]\ar[d]						&& \displaystyle\prod_V\mathscr{M}_V \ar[d]^{\textup{ev}} \\
\displaystyle\prod_{E\in E(\tilde{\Gamma})} D_E \ar[rr]^\delta				&& \displaystyle\prod_V\prod_{\substack{E\in E(\tilde{\Gamma}) \\ V \in E}} D_E
}
\end{xy}
\end{equation*}

By \cite{KLR} there is an \'etale morphism $\textup{cut} : \mathscr{M}_{\tilde{h}} \rightarrow \bigtimes_V\mathscr{M}_V^\circ$ of degree $\textup{deg}(\textup{cut}) = (\prod_{E\in E(\tilde{\Gamma})}w_E)/l_{\tilde{\Gamma}}$, where $l_{\tilde{\Gamma}} = \textup{lcm}\{w_E\}$. By compatibility of obstruction theories (\cite{KLR}, {\S}9) we have
\[ \llbracket\mathscr{M}_{\tilde{h}}\rrbracket = \textup{cut}^\star \delta^!\prod_{V\in V(\tilde{\Gamma})}\llbracket\mathscr{M}_V^\circ\rrbracket. \]
By the projection formula $\textup{cut}_\star\textup{cut}^\star$ is multiplication with $\textup{deg}(\textup{cut})$, so
\[ \textup{cut}_\star\llbracket\mathscr{M}_{\tilde{h}}\rrbracket = \frac{1}{\ell_{\tilde{\Gamma}}}\prod_{E\in E(\tilde{\Gamma})}w_E \cdot \delta^!\prod_{V\in V(\tilde{\Gamma})}\llbracket\mathscr{M}_V^\circ\rrbracket. \]

\begin{prop}[Gluing formula]
\label{prop:gluing}
\[ \int_{\llbracket\mathscr{M}_{\tilde{h}}\rrbracket}1 = \frac{1}{\ell_{\tilde{\Gamma}}}\prod_{E\in E(\tilde{\Gamma})}w_E \cdot \int_{\delta^!\prod_V\llbracket\mathscr{M}_V\rrbracket}1. \]
\end{prop}

\begin{proof}
By the above formula, the cycles $\textup{cut}_\star\llbracket\mathscr{M}_{\tilde{h}}\rrbracket$ and $\frac{1}{\ell_{\tilde{\Gamma}}}\prod_{E\in E(\tilde{\Gamma})}w_E \cdot \delta^!\prod_V\llbracket\mathscr{M}_V\rrbracket$ have the same restriction to the open substack $\bigtimes_V\mathscr{M}_V^\circ$ of $\bigtimes_V\mathscr{M}_V$. Hence their difference is rationally equivalent to a cycle supported on the  closed substack $Z:=(\bigtimes_V\mathscr{M}_V)\setminus(\bigtimes_V\mathscr{M}_V^\circ)$. Suppose there exists an element $(f_V : C_V \rightarrow Y_{\tilde{h}(V)})_{V\in V(\tilde{\Gamma})}\in Z$. Then at least one of the source curves $C_V$ would contain a nontrivial cycle of components as can be seen by a loop construction as in the proof of \cite{GPS}, Proposition 4.2, or \cite{Gra}, Lemma 4.5. This contradicts $g=0$, so $Z$ is empty, completing the proof.
\end{proof}

\begin{prop}[Unique gluing]
\label{prop:unique}
\[ \int_{\delta^!\prod_{V}\llbracket\mathscr{M}_V\rrbracket} 1 = \prod_{V\in V(\tilde{\Gamma})}N_V. \]
\end{prop}

\begin{proof}
This is similar to the proof of \cite{Gra}, Proposition 4.13. As before, $\tilde{\Gamma}$ is a rooted tree with root vertex $V_{\textup{out}}$. This gives an orientation of the edges of $\tilde{\Gamma}$. The only gluing that gives a nonzero contribution after integration is the one according to this orientation. Any other gluing gives a negative virtual dimension, since we have too many conditions on one of the irreducible components.
\end{proof}

\begin{thm}[Degeneration formula]
\label{thm:deg}
Let $P$ be a point in an unbounded chamber of $\mathscr{S}_{p+q}$. Then
\[ N_{p,q}(X,\beta) = \sum_{\tilde{h}\in\tilde{\mathfrak{H}}_{p,q}^\beta(P)} \frac{1}{|\textup{Aut}(\tilde{h})|} \cdot \prod_{E\in E(\tilde{\Gamma})}w_E \cdot \prod_{V\in V(\tilde{\Gamma})} N_V. \]
\end{thm}

\begin{proof}
Since the virtual dimension of $\mathscr{M}_\beta$ is zero, integration (i.e., proper pushforward to a point) of the decomposition formula (Proposition \ref{prop:dec}) gives
\[ N_{p,q}(X,\beta) = \sum_{\tilde{h}\in\tilde{\mathfrak{H}}_{p,q}^\beta(P)} \frac{1}{|\textup{Aut}(\tilde{h})|}\int_{\llbracket\mathscr{M}_{\tilde{h}}\rrbracket} 1. \]
Using Propositions \ref{prop:gluing} and \ref{prop:unique} we get the above formula.
\end{proof}

We get a more symmetric formula by summing over balanced tropical curves:

\begin{defi}
\label{defi:Ntor}
Let $h : \Gamma \rightarrow B$ be a tropical curve in $\mathfrak{H}_\beta$ and let $V$ be a vertex of $\Gamma$. Then the image of $V$ under the map from Lemma \ref{lem:tilde} is a vertex of $\tilde{\Gamma}$ of type (I) or (III). Let $\textbf{m}$ and $\textbf{w}$ be as in the respective case of Proposition \ref{prop:N} and define 
$ N_V^{\textup{tor}} := N_{\textbf{m}}(\textbf{w}). $
Note that $N_V^{\textup{tor}}=N_V$ for vertices of type (I).
\end{defi}

\begin{defi}
\label{defi:Nh}
Define $N_{p,q}^{\textup{trop}}(X,\beta)=\sum_{h\in\mathfrak{H}_{p,q}^\beta(P)} N_h$, where $P$ is a point in an unbounded chamber of $\mathscr{S}_{p+q}$ and
\[ N_h := \frac{1}{|\textup{Aut}(h)|} \cdot \prod_{E\in E(\Gamma)}w_E\cdot \prod_{E\in L_\Delta(\Gamma)}\frac{(-1)^{w_E-1}}{w_E}\cdot\prod_{V\in V(\Gamma)} N_V^{\textup{tor}}. \]
Here $L_\Delta(\Gamma)$ is the set of bounded legs of $\Gamma$. 
\end{defi}

Recall the definition of tropical mutliplicity from Definition \ref{defi:mult}.

\begin{prop}
For a tropical curve $h : \Gamma \rightarrow B$ in $\mathfrak{H}_q$ we have
\[ N_h = \textup{Mult}(h) \]
\end{prop}

\begin{proof}
By the tropical correspondence theorem with point conditions on toric divisors (\cite{GPS}, Theorem 3.4) we have $m_V = \prod_{E\rightarrow V}w_E \cdot N_V^{\textup{tor}}$. The product is over all eges of $\Gamma$ pointing towards $V$ with respect to the orientation of $\Gamma$ such that all edges point towards the root vertex $V_{\textup{out}}$. Then 
\[ \prod_{V\in V(\Gamma)}m_V = \prod_{E\in E(\Gamma)} w_E \cdot \prod_{V\in V(\Gamma)} N_V^{\textup{tor}}, \]
as each $E\in E(\Gamma)$ occurs exactly once. Plugging this and $m_E=(-1)^{w_E+1}/w_E^2$ into the definition of $\textup{Mult}(h)$ we obtain $N_h$.
\end{proof}

\begin{thm}[Tropical correspondence theorem]
\label{thm:degmax}
\[ N_{p,q}^{\textup{trop}}(X,\beta) = pN_{p,q}(X,\beta). \]
\end{thm}

\begin{proof}
This is a simple rearrangement of equations as in \cite{Gra}, Theorem 4.17.
\end{proof}

\subsection{Broken line calculations}								

\begin{defi}
Let $\mathfrak{p}\in\mathscr{S}_\infty$ be a wall and choose $x\in\textup{Int}(\mathfrak{p})$. Define $\mathfrak{H}_{\mathfrak{p},w}^\circ$ to be the set of all tropical disks $h^\circ : \Gamma \rightarrow B$ with $h^\circ(V_\infty)=x$ and $u_{(V_\infty,E_\infty)}=-w\cdot m_{\mathfrak{p}}$. Note that the sets $\mathfrak{H}_{\mathfrak{p},w}$ are in bijection for different choices of $x\in\textup{Int}(\mathfrak{p})$. For $h^\circ\in \mathfrak{H}_{\mathfrak{p},w}^\circ$ define $\textup{Mult}(h^\circ)$ as in Definition \ref{defi:Nh}.
\end{defi}

\begin{lem}[\cite{Gra}, Proposition 5.20]
\label{lem:scattering}
For a wall $\mathfrak{p}$ of $\mathscr{S}_\infty$ we have
\[ \textup{log }f_{\mathfrak{p}} = \sum_{w=1}^\infty\sum_{h^\circ\in\mathfrak{H}_{\mathfrak{p},w}^\circ} w\textup{Mult}(h^\circ) z^{(wm_{\mathfrak{p}},0)}. \]
\end{lem}

\begin{prop}[\cite{CPS}, Proposition 6.15]
\label{prop:broken}
\[ a_{\mathfrak{b}} = \sum_{h\in\mu^{-1}(\mathfrak{b})} \textup{Mult}(h). \]
\end{prop}

\begin{proof}
Consider a break point $P'$ of $\mathfrak{b}$ and let $az^m, a'z^{m'}$ be the monomials attached to the line segments before and after the breaking, respectively. To complete $\mathfrak{b}$ (Proposition \ref{prop:bhcirc}) at $P'$ we have to add one or more tropical disks ending in $P'$ such that the sum of their ending weights $u_{(V_\infty,E_\infty)}$ equals $\widebar{m}+\widebar{m}'$. Let $\mathfrak{H}_{\mathfrak{b}}(P')$ be the set of such collections and for $(h_1^\circ,\ldots,h_s^\circ)\in\mathfrak{H}_{\mathfrak{b}}(P')$ let $m_{V_\infty}$ be the multiplicity of the vertex mapping to $P'$ for any completion $h\in\mu^{-1}(\mathfrak{b})$ involving $(h_1^\circ,\ldots,h_s^\circ)$. This is well-defined, since $m_{V_\infty}$ depends only on $(h_1^\circ,\ldots,h_s^\circ)$ and not on tropical disks we add to the other break points of $\mathfrak{b}$.

\underline{Claim}: 
\[ \frac{a'}{a} = \sum_{(h_1^\circ,\ldots,h_s^\circ)\in\mathfrak{H}_{\mathfrak{b}}(P')} m_{V_\infty}\prod_{i=1}^s \text{Mult}(h_i^\circ). \]
Let $\mathfrak{p}$ be the wall containing the break point $P'$ and define $a_{\mathfrak{p}}=|\widebar{m}\wedge\widebar{m}'|\text{log }f_{\mathfrak{p}}$. By definition $\frac{a'}{a}$ is the coefficient of $z^{m-m'}$ in $\text{exp}(a_{\mathfrak{p}})$. By Lemma \ref{lem:scattering}, $a_{\mathfrak{p}}$ is the generating function of $|\widebar{m}\wedge\widebar{m}'|w_{E_\infty}\textup{Mult}(h^\circ)$ for tropical disks ending in $P'$. Hence, by standard combinatorial arguments, $\text{exp}(a_{\mathfrak{p}})$ is a sum over collections of such tropical disks. Hence, the coefficient of $z^{m-m'}$ is given by
\[ \frac{a'}{a} = \sum_{(h_1^\circ,\ldots,h_s^\circ)\in\mathfrak{H}_{\mathfrak{b}}(P')}\prod_{i=1}^s |\widebar{m}\wedge\widebar{m}'|w_{E_\infty}(h_s^\circ) \cdot \textup{Mult}(h_s^\circ). \]
But $|\widebar{m}\wedge\widebar{m}'|\prod_{i=1}^s w_{E_\infty}(h_s^\circ)$ is nothing but $m_{V_\infty}$. This proves the claim.

Now let $a_1z^{m_1},\ldots,a_rz^{m_r}$ be all the monomials attached to line segments of $\mathfrak{b}$ and for $i=1,\ldots,r-1$ let $P_i$ be the break point between $a_iz^{m_i}$ and $a_{i+1}z^{m_{i+1}}$. We can expand
\[ a_{\mathfrak{b}} := a_r = \frac{a_r}{a_{r-1}}\frac{a_{r-1}}{a_{r-2}}\cdots\frac{a_2}{a_1}a_1. \]
By definition $a_1=1$ and for each fraction we can use the claim to obtain
\[ a_{\mathfrak{b}} = \prod_{i=1}^{r-1} \sum_{(h_1^\circ,\ldots,h_s^\circ)\in\mathfrak{H}_{\mathfrak{b}}(P_i)} m_{V_\infty}\prod_{i=1}^s\text{Mult}(h_i^\circ). \]
The first product and summation can be replaced by a summation over all possible combinations of collections $(h_1^\circ,\ldots,h_s^\circ)$ for all $P_i$. But this is nothing but a choice of completion $h\in\mu^{-1}(\mathfrak{b})$. Moreover, the product of the $m_{V_\infty}\prod_{i=1}^s\text{Mult}(h_i^\circ)$ is nothing but $\text{Mult}(h)$. Hence, we get the formula claimed in the proposition.
\end{proof}

\section{Theta functions}										
\label{S:theta}

\begin{defi}
\label{defi:theta}
For a point $P\in B_0$ and an asymptotic direction $m$ define the corresponding \textit{theta function} by
\[ \vartheta_m(P) = \sum_{\mathfrak{b}\in\mathfrak{B}_m(P)} a_{\mathfrak{b}} z^{m_{\mathfrak{b}}} \]
\end{defi}

In our case, with smooth divisor $D$, the dual intersection complex $B$ has exactly one unbounded direction $m_{\textup{out}}$, so asymptotic directions on $B$ are just multiples of $m_{\textup{out}}$. We write $\vartheta_q(P)$ for $\vartheta_{q\cdot m_{\textup{out}}}(P)$ with $d\in \mathbb{N}$.

\begin{prop}[\cite{GHS}, Theorem 3.24, \cite{GSintrinsic}, Theorem 1.9]
Theta functions generate a commutative ring (associative if $X$ is Fano) with unit $\vartheta_0$ by the multiplication rule
\[ \vartheta_p(P) \cdot \vartheta_q(P) = \sum_{r=0}^\infty \alpha_{p,q}^r(P) \vartheta_r(P) \]
with structure constants
\[ \alpha_{p,q}^r(P) = \sum_{\substack{(\mathfrak{b}_1,\mathfrak{b}_2)\in\mathfrak{B}_p(P)\times\mathfrak{B}_q(P) \\ m_{\mathfrak{b}_1}+m_{\mathfrak{b}_2}= \ r}} a_{\mathfrak{b}_1}a_{\mathfrak{b}_2} \]
\end{prop}

\begin{thm}[Theorems \ref{thm:main} and \ref{thm:main2}]
\label{thm:theta}
Write $x=z^{(-m_{\textup{out}},-1)}$ and $t=z^{(0,1)}$. Then
\begin{equation}
\label{eq:1}
\vartheta_q = x^{-q} + \sum_{p=1}^\infty pN_{p,q}t^{p+q}x^p
\end{equation}
Moreover, $\alpha_{p,q}^{r}=1$ if $r=p+q$ and otherwise
\begin{equation}
\label{eq:2}
\alpha_{p,q}^r = ((p-r)N_{p-r,q} + (q-r)N_{q-r,p})t^{p+q-r},
\end{equation}
where we define $N_{p,q}=0$ whenever $p\leq 0$.
\end{thm}

\begin{proof}
This follows from Theorem \ref{thm:degmax} and Proposition \ref{prop:broken}.
\end{proof}

Plugging both expressions of Theorem \ref{thm:theta} into the multiplication rule we obtain relations among the $N_{p,q}$. Since powers of $\vartheta_1$ generate the theta ring these equations determine all $N_{p,q}$ by only knowing the invariants $N_{p,1}$ or, equivalently, the invariants $N_{1,q}$.

\begin{expl}
We use Theorem \ref{thm:theta} to obtain relations among the $2$-marked invariants $N_{p,q}$ for $\mathbb{P}^2$ up to order $d=(p+q)/3=2$. By \eqref{eq:1} we have 
\begin{eqnarray*}
\vartheta_1&=&x^{-1}+2N_{2,1}t^3x^2+5N_{5,1}t^6x^5+\mathcal{O}(t^9) \\
\vartheta_2&=&x^{-2}+N_{1,2}t^3x+4N_{4,2}t^6x^4+\mathcal{O}(t^9) \\
\vartheta_3&=&x^{-3}+3N_{3,3}t^6x^3+\mathcal{O}(t^9) \\
\vartheta_4&=&x^{-4}+2N_{2,4}t^6x^2+\mathcal{O}(t^9) \\
\vartheta_5&=&x^{-5}+N_{1,5}t^6x+\mathcal{O}(t^9)
\end{eqnarray*}
By direct mutliplication we get
\[ \vartheta_1\cdot\vartheta_1=x^{-2}+4N_{2,1}t^3x+(4N_{2,1}^2+10N_{5,1})t^6x^4+\mathcal{O}(t^9). \]
On the other hand, \eqref{eq:2} gives, with $\alpha_{1,1}^1=0$,
\[ \vartheta_1\cdot\vartheta_1 = \vartheta_2 = x^{-2}+N_{1,2}t^3x+4N_{4,2}t^6x^4 + \mathcal{O}(t^9) \]
Comparing these two equations we get the relations
\[ N_{1,2}=4N_{2,1}, \qquad 2N_{4,2}=2N_{2,1}^2+5N_{5,1} \]
Similarly, comparing
\[ \vartheta_1\cdot\vartheta_2=x^{-3}+(N_{1,2}+2N_{2,1})t^3x^0+(2N_{1,2}N_{2,1}+4N_{4,2}+5N_{5,1})t^6x^3+\mathcal{O}(t^9) \]
with
\[ \vartheta_1\cdot\vartheta_2=\vartheta_3+(N_{1,2}+2N_{2,1})t^3\vartheta_0=x^{-3}+(N_{1,2}+2N_{2,1})x^0+3N_{3,3}x^3+\mathcal{O}(t^9) \]
we get
\[ 3N_{3,3} = 2N_{1,2}N_{2,1}+4N_{4,2}+5N_{5,1}. \]
Comparing
\[ \vartheta_1\cdot \vartheta_3=x^{-4}+2N_{2,1}t^3x^{-1}+(3N_{3,3}+5N_{5,1})t^6x^2+\mathcal{O}(t^9) \]
with
\[ \vartheta_1\cdot \vartheta_3=\vartheta_4+2N_{2,1}t^3\vartheta_1=x^{-4}+2N_{2,1}t^3x^{-1}+(4N_{2,1}^2+2N_{2,4})t^6x^2+\mathcal{O}(t^9) \]
we get the relation
\[ 3N_{3,3}+5N_{5,1}=4N_{2,1}^2+2N_{2,4} \]
and comparing
\[ \vartheta_1\cdot\vartheta_4 = x^{-5}+2N_{2,1}t^3x^{-2}+(2N_{2,4}+5N_{5,1})t^6x+\mathcal{O}(t^9) \]
with
\[ \vartheta_1\cdot\vartheta_4 = \vartheta_5+2N_{2,1}t^3\vartheta_2 = x^{-5}+2N_{2,1}t^3x^{-2}+(N_{1,5}+2N_{1,2}N_{2,1})t^6x+\mathcal{O}(t^9) \]
we get
\[ 2N_{2,4}+5N_{5,1}=N_{1,5}+2N_{1,2}N_{2,1}. \]
Knowing $N_{1,2}=4$ and $N_{1,5}=25$, e.g. by direct computation as in \S\ref{S:calc}, we can solve the above equations and get $N_{2,1}=1$ as well as
\[ N_{2,4}=14,\qquad N_{3,3}=9,\qquad N_{4,2}=\frac{7}{2},\qquad N_{5,1}=1. \]
\end{expl}

\section{Higher genus and $\boldsymbol{q}$-refined invariants}						
\label{S:genus}

For an effective curve class $\underline{\beta}$ of $X$ let $\beta_{p,q}^g$ be the class of stable log maps to $X$ of genus $g$, class $\underline{\beta}$ and two marked points with contact orders $p$ and $q$ with $D$. The moduli space $\mathscr{M}(X,\beta_{p,q}^g)$ of basic stable log maps of class $\beta_{p,q}^g$ has virtual dimension $g+1$. We cut this dimension down to zero by fixing the image of the first marked point and inserting a \textit{lambda class}. Let $\pi : \mathcal{C} \rightarrow \mathscr{M}(X,\beta_{p,q}^g)$ be the universal curve, of relative dualizing sheaf $\omega_\pi$. Then $\mathbb{E}=\pi_\star\omega_\pi$ is a rank $g$ vector bundle over $\mathscr{M}(X,\beta^g)$, called the Hodge bundle. The lambda classes are the Chern classes of the Hodge bundle, $\lambda_j=c_j(\mathbb{E})$. Let $\textup{ev} : \mathscr{M}(X,\beta_{p,q}^g) \rightarrow D$ be the evaluation map at the marked point of order $p$. Define the $2$-marked log Gromov-Witten invariant
\[ N_{p,q}^g(X,\beta) = \int_{\llbracket\mathscr{M}(X,\beta_{p,q}^g)\rrbracket} (-1)^g\lambda_g \textup{ev}^\star[\textup{pt}] \in \mathbb{Q}. \]

\begin{defi}
Let $h : \Gamma \rightarrow B$ be a tropical curve. For a trivalent vertex $V$ with multiplicity $m_V$ (Definition \ref{defi:mult}) define, with $\boldsymbol{q}=e^{i\hbar}$,
\[ m_V(\boldsymbol{q}) = \frac{1}{i\hbar}\left(\boldsymbol{q}^{m_V/2}-\boldsymbol{q}^{-m_V/2}\right). \]
For a vertex with higher valvency define $m_V(\boldsymbol{q}) = \prod_{V'\in V(h'_V)} m_{V'}(\boldsymbol{q})$ with $h'_V$ as in Definition \ref{defi:mult}. For a bounded leg $E$ with weight $w_E$ define
\[ m_E(\boldsymbol{q}) = \frac{(-1)^{w_E}}{w_E}\cdot \frac{i\hbar}{\boldsymbol{q}^{w_E/2}-\boldsymbol{q}^{-w_E/2}}. \]
Then define the \textit{$\boldsymbol{q}$-refined multiplicity} of $h$ to be
\[ m_h(\boldsymbol{q}) = \frac{1}{|\textup{Aut}(h)|} \cdot \prod_{V\in V(\Gamma)} m_V(\boldsymbol{q}) \cdot \prod_{E\in L_\Delta(\Gamma)} m_E(\boldsymbol{q}). \]
\end{defi}

\begin{thm}
Let $P$ be a point in an unbounded chamber of $\mathscr{S}_{p+q}$. Then, with $\boldsymbol{q}=e^{i\hbar}$,
\[ \sum_{g\geq 0} N_{p,q}^g(X,\beta)\hbar^{2g} = \sum_{h\in\mathfrak{H}_{p,q}^\beta(P)} m_h(\boldsymbol{q}) \]
\end{thm}

\begin{proof}
Consider a stable log map in $\mathscr{M}(X,\beta^g)$ and let $h : \Gamma \rightarrow B$ be its tropicalization. The genus of $h$ is $g_h = g_\Gamma + \sum_V g_V$, where $g_\Gamma$ is the genus of the graph $\Gamma$ and $g_V$ is the genus attached to a vertex $V$. Using gluing and vanishing properties of lambda classes, Bousseau showed in \cite{Bou1} that $\Gamma$ is still a tree ($g_\Gamma=0$), i.e., all contributions to $g_h$ come from vertices. Hence, $h$ maps to an element of $\mathfrak{H}_\beta$ by forgetting genera at vertices $g_V$. So we can sum over $\mathfrak{H}_{p,q}^\beta(P)$, but have to consider $\boldsymbol{q}$-refined contributions of vertices. By \cite{Bou1}, Proposition 29, the contribution of a vertex $V$ with classical multiplicity $m_V$ is $m_V(\boldsymbol{q})$. By \cite{Bou2}, Proposition 5.1, the contribution of a bounded leg $L$ is $m_L(\boldsymbol{q})$.
\end{proof}

To obtain a higher genus version of Theorem \ref{thm:main}, we have to $\boldsymbol{q}$-refine the slab functions in the initial wall structure $\mathscr{S}_0$. For $\boldsymbol{q}$-refined wall structures it turns out to be more convenient to work with the logarithm of such functions. Define the $\boldsymbol{q}$-refined initial wall structure $\mathscr{S}_0(\boldsymbol{q})$ to have the same slabs as $\mathscr{S}_0$ but with slab functions $f_{\mathfrak{p}}=1+z^{(m,0)}$ replaced by $f_{\mathfrak{p}}(\boldsymbol{q})$, where
\[ \textup{log }f_{\mathfrak{p}}(\boldsymbol{q}) = \sum_{k\geq 1}\frac{(-1)^{k+1}i\hbar}{\boldsymbol{q}^{k/2}-\boldsymbol{q}^{-k/2}}z^{(km,0)}. \]
The coefficient of $z^{(km,0)}$ is the $\boldsymbol{q}$-multiplicity of a bounded leg of weight $k$. 

The inductive construction of wall structures and broken lines can be performed with $\textbf{q}$-refined objects. The formulas are the same, the only difference being that all functions carry an additional $\hbar$-variable. Hence, for each finite order $k>0$ we obtain a $\textbf{q}$-refined wall structure $\mathscr{S}_k(\boldsymbol{q})$. Let $\mathscr{S}_\infty(\boldsymbol{q})$ denote the limit $k\rightarrow\infty$.

Let $P$ be a point in an unbounded chamber of $\mathscr{S}_{p+q}(\boldsymbol{q})$. Then a broken line $\mathfrak{b}\in\mathfrak{B}_{p,q}(P)$ has ending monomial $a_{\mathfrak{b}}(\boldsymbol{q})t^{p+q}x^p$ for $x=z^{(-m_{\text{out}},-1)}$. Define $\vartheta_q(\boldsymbol{q})$ and $\alpha_{p,q}^r(\boldsymbol{q})$ similar to $\vartheta_q$ and $\alpha_{p,q}^r$ in {\S}\ref{S:theta}, but with $\boldsymbol{q}$-refined broken lines.

\begin{thm}[Higher genus version of Theorems \ref{thm:main} and \ref{thm:main2}]
\text{ } \\
Write $x=z^{(-m_{\textup{out}},-1)}$ and $t=x^{(0,1)}$. Then
\begin{equation}
\label{eq:1}
\vartheta_q(\boldsymbol{q}) = x^{-q} + \sum_{p=1}^\infty \sum_{g\geq 0} pN_{p,q}^g \hbar^{2g}t^{p+q}x^p
\end{equation}
Moreover, $\alpha_{p,q}^{r}(\boldsymbol{q})=1$ if $r=p+q$ and otherwise
\begin{equation}
\label{eq:2}
\alpha_{p,q}^r(\boldsymbol{q}) = \sum_{g\geq 0}((p-r)N_{p-r,q}^g + (q-r)N_{q-r,p}^g)\hbar^{2g}t^{p+q-r},
\end{equation}
where we define $N_{p,q}^g=0$ whenever $p\leq 0$.
\end{thm}

\begin{proof}
By \cite{Gra}, we have $\textup{log }f_{\mathfrak{p}}(\boldsymbol{q}) = \sum_{w=1}^\infty\sum_{h\in\mathfrak{H}_{\mathfrak{p},w}} m_h(\boldsymbol{q}) z^{(wm_{\mathfrak{p}},0)}$, the $\boldsymbol{q}$-refined version of Lemma \ref{lem:scattering}. As a consequence, we get a $\boldsymbol{q}$-refined version of Proposition \ref{prop:broken}: $a_{\mathfrak{b}}(\boldsymbol{q}) = \sum_{h\in\mu^{-1}(\mathfrak{b})} m_h(\boldsymbol{q})$. Now the statement follows from the definitions of $\vartheta_q(\boldsymbol{q})$ and $\alpha_{p,q}^r(\boldsymbol{q})$.
\end{proof}

\section{Example calculations}									
\label{S:calc}

We use a sage code to calculate the numbers $N_{1,3d-1}(\mathbb{P}^2,dL)$ for $d\leq 4$. It can be found on \href{timgraefnitz.com}{timgraefnitz.com}. In the code, broken lines are implemented in reversed order. We start with point $P$ on an unbounded maximal cell of $\mathscr{S}_{3d}(\mathbb{P}^2)$ and a broken line coming out of this point in the negative of the unique unbounded direction $m_{\textup{out}}$, with attached monomial $z^{qm_{\textup{out}}}$. We can do this, since all broken lines that end in $P$ have to be parallel to $m_{\textup{out}}$. When the broken line hits a wall, we apply the transformation $z^m \mapsto f^{\braket{n,\widebar{m}}}z^m$, where $n$ is the normal direction of the wall. Each term in $f^{\braket{n,\widebar{m}}}z^m$ gives a new possible broken line. The above procedure is applied recursively until the direction of the new broken line is $m_{\textup{out}}$. Then we have found a broken line with asymptotic monomial $z^{qm_{\textup{out}}}$ and ending in $P$. If we add together the coefficients $a_{\mathfrak{b}}$ of the broken lines $\mathfrak{b}$ with asymptotic monomial $z^{qm_{\textup{out}}}$ and resulting monomial $a_{\mathfrak{b}}z^{-pm_{\textup{out}}}$ we get the tropical count $N_{p,q}^{\textup{trop}}(\mathbb{P}^2,dL)$, where $d=(p+q)/3$. Using the tropical correspondence (Theorem \ref{thm:degmax}) we obtain the $2$-marked log Gromov-Witten invariants $N_{p,q}=N_{p,q}(\mathbb{P}^2,dL)$. Figure \ref{fig:sage} shows the broken lines for $d=2$. This gives
\vspace{3mm}
\begin{center}
\renewcommand{\arraystretch}{1.2}
\begin{tabular}{|c|c|c|c|c|} 								\hline
$N_{1,5}=25$	& $N_{2,4}=14$		& $N_{3,3}=9$	& $N_{4,2}=\frac{7}{2}$		& $N_{5,1}=1$	\\ \hline
\end{tabular}
\end{center}

\begin{figure}[h!]
\centering
\begin{tikzpicture}[scale=0.8,rotate=90]
\clip (-3,-1) rectangle (6.6,2);
\draw[dashed] (0.0, 0.5) -- (-1.0, 1.0);
\draw[dashed] (0.0, 0.5) -- (-1.0, 0.0);
\draw[dashed] (-1.5, -0.5) -- (-1.0, 0.0);
\draw[dashed] (-1.5, 1.5) -- (-1.0, 1.0);
\draw[dashed] (-1.5, 1.5) -- (-4.0, 2.0);
\draw[dashed] (-1.5, -0.5) -- (-4.0, -1.0);
\draw[dashed] (-6.0, -1.5) -- (-4.0, -1.0);
\draw[dashed] (-6.0, 2.5) -- (-4.0, 2.0);
\draw[-,gray] (0.0, 0.5) node[fill,cross,inner sep=3pt,rotate=90.0,color=gray]{} -- (0.0, 2.5);
\draw[-,gray] (0.0, 0.5) node[fill,cross,inner sep=3pt,rotate=-90.0,color=gray]{} -- (0.0, -1.5);
\draw[-,gray] (-1.5, 1.5) node[fill,cross,inner sep=3pt,rotate=-18.435,color=gray]{} -- (6.0, -1.0);
\draw[-,gray] (-1.5, -0.5) node[fill,cross,inner sep=3pt,rotate=18.435,color=gray]{} -- (6.0, 2.0);
\draw[-,gray] (-1.5, 1.5) node[fill,cross,inner sep=3pt,rotate=161.565,color=gray]{} -- (-4.5, 2.5);
\draw[-,gray] (-1.5, -0.5) node[fill,cross,inner sep=3pt,rotate=-161.565,color=gray]{} -- (-4.5, -1.5);
\draw[-,gray] (-6.0, 2.5) node[fill,cross,inner sep=3pt,rotate=-9.462,color=gray]{} -- (6.0, 0.5);
\draw[-,gray] (-6.0, -1.5) node[fill,cross,inner sep=3pt,rotate=9.462,color=gray]{} -- (6.0, 0.5);
\draw[-,gray] (-3.0, 2.0) -- (6.0, 2.0);
\draw[-,gray] (0.0, 0.0) -- (6.0, 0.0);
\draw[-,gray] (-3.0, -1.0) -- (6.0, -1.0);
\draw[-,gray] (0.0, 1.0) -- (6.0, 1.0);
\draw[-,gray] (1.5, 0.5) -- (6.0, 0.5);
\draw[-,gray] (3.0, 0.0) -- (6.0, -0.5);
\draw[-,gray] (3.0, 1.0) -- (6.0, 1.5);
\draw[-,gray] (-0.0, -0.0) -- (4.5, -1.5);
\draw[-,gray] (-0.0, -0.0) -- (6.0, -0.0);
\draw[-,gray] (-3.0, 2.0) -- (-3.0, 2.5);
\draw[-,gray] (-0.0, -0.0) -- (6.0, 1.0);
\draw[-,gray] (-0.0, 1.0) -- (4.5, 2.5);
\draw[-,gray] (-0.0, 1.0) -- (6.0, 0.0);
\draw[-,gray] (-0.0, 1.0) -- (6.0, 1.0);
\draw[-,gray] (-3.0, -1.0) -- (-3.0, -1.5);
\fill[black!40!green] (5.9, 0.382) circle (2pt);
\draw[black!40!green,line width=1pt] (5.9, 0.382) -- (1.146, 0.382);
\draw[black!40!green,line width=1pt] (1.146, 0.382) -- (1.429, 0.524);
\draw[black!40!green,line width=1pt] (1.429, 0.524) -- (6, 0.524) node[above]{$5$};
\fill[black!40!green] (5.9, 0.382) circle (2pt);
\draw[black!40!green,line width=1pt] (5.9, 0.382) -- (0.0, 0.382);
\draw[black!40!green,line width=1pt] (0.0, 0.382) -- (-0.573, -0.191);
\draw[black!40!green,line width=1pt] (-0.573, -0.191) -- (0.0, -0.076);
\draw[black!40!green,line width=1pt] (0.0, -0.076) -- (6, -0.076) node[above]{$5$};
\fill[black!40!green] (5.9, 0.382) circle (2pt);
\draw[black!40!green,line width=1pt] (5.9, 0.382) -- (0.0, 0.382);
\draw[black!40!green,line width=1pt] (0.0, 0.382) -- (-1.058, -0.676);
\draw[black!40!green,line width=1pt] (-1.058, -0.676) -- (6, -0.676) node[above]{$5$};
\fill[black!40!green] (5.9, 0.382) circle (2pt);
\draw[black!40!green,line width=1pt] (5.9, 0.382) -- (-0.236, 0.382);
\draw[black!40!green,line width=1pt] (-0.236, 0.618) -- (-0.927, 1.309);
\draw[black!40!green,line width=1pt] (-0.927, 1.309) -- (0.0, 1.124);
\draw[black!40!green,line width=1pt] (0.0, 1.124) -- (6, 1.124) node[above]{$5$};
\fill[black!40!green] (5.9, 0.382) circle (2pt);
\draw[black!40!green,line width=1pt] (5.9, 0.382) -- (-0.236, 0.382);
\draw[black!40!green,line width=1pt] (-0.236, 0.618) -- (-1.342, 1.724);
\draw[black!40!green,line width=1pt] (-1.342, 1.724) -- (6, 1.724) node[above]{$5$};
\draw (-2.7,0.5) node{$N_{1,5}^{\textup{trop}}=25$};
\end{tikzpicture}
\hspace{2mm}
\begin{tikzpicture}[scale=0.8,rotate=90]
\clip (-3,-1) rectangle (6.6,2);
\draw[dashed] (0.0, 0.5) -- (-1.0, 1.0);
\draw[dashed] (0.0, 0.5) -- (-1.0, 0.0);
\draw[dashed] (-1.5, -0.5) -- (-1.0, 0.0);
\draw[dashed] (-1.5, 1.5) -- (-1.0, 1.0);
\draw[dashed] (-1.5, 1.5) -- (-4.0, 2.0);
\draw[dashed] (-1.5, -0.5) -- (-4.0, -1.0);
\draw[dashed] (-6.0, -1.5) -- (-4.0, -1.0);
\draw[dashed] (-6.0, 2.5) -- (-4.0, 2.0);
\draw[-,gray] (0.0, 0.5) node[fill,cross,inner sep=3pt,rotate=90.0,color=gray]{} -- (0.0, 2.5);
\draw[-,gray] (0.0, 0.5) node[fill,cross,inner sep=3pt,rotate=-90.0,color=gray]{} -- (0.0, -1.5);
\draw[-,gray] (-1.5, 1.5) node[fill,cross,inner sep=3pt,rotate=-18.435,color=gray]{} -- (6.0, -1.0);
\draw[-,gray] (-1.5, -0.5) node[fill,cross,inner sep=3pt,rotate=18.435,color=gray]{} -- (6.0, 2.0);
\draw[-,gray] (-1.5, 1.5) node[fill,cross,inner sep=3pt,rotate=161.565,color=gray]{} -- (-4.5, 2.5);
\draw[-,gray] (-1.5, -0.5) node[fill,cross,inner sep=3pt,rotate=-161.565,color=gray]{} -- (-4.5, -1.5);
\draw[-,gray] (-6.0, 2.5) node[fill,cross,inner sep=3pt,rotate=-9.462,color=gray]{} -- (6.0, 0.5);
\draw[-,gray] (-6.0, -1.5) node[fill,cross,inner sep=3pt,rotate=9.462,color=gray]{} -- (6.0, 0.5);
\draw[-,gray] (-3.0, 2.0) -- (6.0, 2.0);
\draw[-,gray] (0.0, 0.0) -- (6.0, 0.0);
\draw[-,gray] (-3.0, -1.0) -- (6.0, -1.0);
\draw[-,gray] (0.0, 1.0) -- (6.0, 1.0);
\draw[-,gray] (1.5, 0.5) -- (6.0, 0.5);
\draw[-,gray] (3.0, 0.0) -- (6.0, -0.5);
\draw[-,gray] (3.0, 1.0) -- (6.0, 1.5);
\draw[-,gray] (-0.0, -0.0) -- (4.5, -1.5);
\draw[-,gray] (-0.0, -0.0) -- (6.0, -0.0);
\draw[-,gray] (-3.0, 2.0) -- (-3.0, 2.5);
\draw[-,gray] (-0.0, -0.0) -- (6.0, 1.0);
\draw[-,gray] (-0.0, 1.0) -- (4.5, 2.5);
\draw[-,gray] (-0.0, 1.0) -- (6.0, 0.0);
\draw[-,gray] (-0.0, 1.0) -- (6.0, 1.0);
\draw[-,gray] (-3.0, -1.0) -- (-3.0, -1.5);
\fill[black!40!green] (5.9, 0.382) circle (2pt);
\draw[black!40!green,line width=1pt] (5.9, 0.382) -- (1.146, 0.382);
\draw[black!40!green,line width=1pt] (1.146, 0.382) -- (1.323, 0.559);
\draw[black!40!green,line width=1pt] (1.323, 0.559) -- (6, 0.559) node[above]{$8$};
\fill[black!40!green] (5.9, 0.382) circle (2pt);
\draw[black!40!green,line width=1pt] (5.9, 0.382) -- (0.0, 0.382);
\draw[black!40!green,line width=1pt] (0.0, 0.382) -- (-1.236, -0.236);
\draw[black!40!green,line width=1pt] (-2.82, -0.764) -- (-2.292, -0.764);
\draw[black!40!green,line width=1pt] (-2.292, -0.764) -- (-2.646, -0.941);
\draw[black!40!green,line width=1pt] (-2.646, -0.941) -- (6, -0.941);
\draw[black!40!green] (6.01,-0.9) node[above]{$8$};
\fill[black!40!green] (5.9, 0.382) circle (2pt);
\draw[black!40!green,line width=1pt] (5.9, 0.382) -- (0.0, 0.382);
\draw[black!40!green,line width=1pt] (0.0, 0.382) -- (-0.573, -0.191);
\draw[black!40!green,line width=1pt] (-0.573, -0.191) -- (6, -0.191) node[above]{$6$};
\fill[black!40!green] (5.9, 0.382) circle (2pt);
\draw[black!40!green,line width=1pt] (5.9, 0.382) -- (-0.236, 0.382);
\draw[black!40!green,line width=1pt] (-0.236, 0.618) -- (-0.927, 1.309);
\draw[black!40!green,line width=1pt] (-0.927, 1.309) -- (6, 1.309) node[above]{$6$};
\draw (-2.7,0.5) node{$N_{2,4}^{\textup{trop}}=28$};
\end{tikzpicture}
\hspace{2mm}
\begin{tikzpicture}[scale=0.8,rotate=90]
\clip (-3,-1) rectangle (6.6,2);
\draw[dashed] (0.0, 0.5) -- (-1.0, 1.0);
\draw[dashed] (0.0, 0.5) -- (-1.0, 0.0);
\draw[dashed] (-1.5, -0.5) -- (-1.0, 0.0);
\draw[dashed] (-1.5, 1.5) -- (-1.0, 1.0);
\draw[dashed] (-1.5, 1.5) -- (-4.0, 2.0);
\draw[dashed] (-1.5, -0.5) -- (-4.0, -1.0);
\draw[dashed] (-6.0, -1.5) -- (-4.0, -1.0);
\draw[dashed] (-6.0, 2.5) -- (-4.0, 2.0);
\draw[-,gray] (0.0, 0.5) node[fill,cross,inner sep=3pt,rotate=90.0,color=gray]{} -- (0.0, 2.5);
\draw[-,gray] (0.0, 0.5) node[fill,cross,inner sep=3pt,rotate=-90.0,color=gray]{} -- (0.0, -1.5);
\draw[-,gray] (-1.5, 1.5) node[fill,cross,inner sep=3pt,rotate=-18.435,color=gray]{} -- (6.0, -1.0);
\draw[-,gray] (-1.5, -0.5) node[fill,cross,inner sep=3pt,rotate=18.435,color=gray]{} -- (6.0, 2.0);
\draw[-,gray] (-1.5, 1.5) node[fill,cross,inner sep=3pt,rotate=161.565,color=gray]{} -- (-4.5, 2.5);
\draw[-,gray] (-1.5, -0.5) node[fill,cross,inner sep=3pt,rotate=-161.565,color=gray]{} -- (-4.5, -1.5);
\draw[-,gray] (-6.0, 2.5) node[fill,cross,inner sep=3pt,rotate=-9.462,color=gray]{} -- (6.0, 0.5);
\draw[-,gray] (-6.0, -1.5) node[fill,cross,inner sep=3pt,rotate=9.462,color=gray]{} -- (6.0, 0.5);
\draw[-,gray] (-3.0, 2.0) -- (6.0, 2.0);
\draw[-,gray] (0.0, 0.0) -- (6.0, 0.0);
\draw[-,gray] (-3.0, -1.0) -- (6.0, -1.0);
\draw[-,gray] (0.0, 1.0) -- (6.0, 1.0);
\draw[-,gray] (1.5, 0.5) -- (6.0, 0.5);
\draw[-,gray] (3.0, 0.0) -- (6.0, -0.5);
\draw[-,gray] (3.0, 1.0) -- (6.0, 1.5);
\draw[-,gray] (-0.0, -0.0) -- (4.5, -1.5);
\draw[-,gray] (-0.0, -0.0) -- (6.0, -0.0);
\draw[-,gray] (-3.0, 2.0) -- (-3.0, 2.5);
\draw[-,gray] (-0.0, -0.0) -- (6.0, 1.0);
\draw[-,gray] (-0.0, 1.0) -- (4.5, 2.5);
\draw[-,gray] (-0.0, 1.0) -- (6.0, 0.0);
\draw[-,gray] (-0.0, 1.0) -- (6.0, 1.0);
\draw[-,gray] (-3.0, -1.0) -- (-3.0, -1.5);
\fill[black!40!green] (5.8, 0.382) circle (2pt);
\draw[black!40!green,line width=1pt] (5.8, 0.382) -- (1.146, 0.382);
\draw[black!40!green,line width=1pt] (1.146, 0.382) -- (1.146, 0.618);
\draw[black!40!green,line width=1pt] (1.146, 0.618) -- (6, 0.618) node[above]{$9$};
\fill[black!40!green] (5.8, 0.382) circle (2pt);
\draw[black!40!green,line width=1pt] (5.8, 0.382) -- (0.0, 0.382);
\draw[black!40!green,line width=1pt] (0.0, 0.382) -- (-1.146, -0.382);
\draw[black!40!green,line width=1pt] (-1.146, -0.382) -- (6, -0.382) node[above]{$9$};
\fill[black!40!green] (5.8, 0.382) circle (2pt);
\draw[black!40!green,line width=1pt] (5.8, 0.382) -- (0.0, 0.382);
\draw[black!40!green,line width=1pt] (0.0, 0.382) -- (-0.708, 0.146);
\draw[black!40!green,line width=1pt] (-0.708, 0.854) -- (-1.146, 1.146);
\draw[black!40!green,line width=1pt] (-3.27, 1.854) -- (-1.854, 1.618);
\draw[black!40!green,line width=1pt] (-1.854, 1.618) -- (6, 1.618) node[above]{$9$};
\draw (-2.7,0.5) node{$N_{3,3}^{\textup{trop}}=27$};
\end{tikzpicture}
\hspace{2mm}
\begin{tikzpicture}[scale=0.8,rotate=90]
\clip (-3,-1) rectangle (6.6,2);
\draw[dashed] (0.0, 0.5) -- (-1.0, 1.0);
\draw[dashed] (0.0, 0.5) -- (-1.0, 0.0);
\draw[dashed] (-1.5, -0.5) -- (-1.0, 0.0);
\draw[dashed] (-1.5, 1.5) -- (-1.0, 1.0);
\draw[dashed] (-1.5, 1.5) -- (-4.0, 2.0);
\draw[dashed] (-1.5, -0.5) -- (-4.0, -1.0);
\draw[dashed] (-6.0, -1.5) -- (-4.0, -1.0);
\draw[dashed] (-6.0, 2.5) -- (-4.0, 2.0);
\draw[-,gray] (0.0, 0.5) node[fill,cross,inner sep=3pt,rotate=90.0,color=gray]{} -- (0.0, 2.5);
\draw[-,gray] (0.0, 0.5) node[fill,cross,inner sep=3pt,rotate=-90.0,color=gray]{} -- (0.0, -1.5);
\draw[-,gray] (-1.5, 1.5) node[fill,cross,inner sep=3pt,rotate=-18.435,color=gray]{} -- (6.0, -1.0);
\draw[-,gray] (-1.5, -0.5) node[fill,cross,inner sep=3pt,rotate=18.435,color=gray]{} -- (6.0, 2.0);
\draw[-,gray] (-1.5, 1.5) node[fill,cross,inner sep=3pt,rotate=161.565,color=gray]{} -- (-4.5, 2.5);
\draw[-,gray] (-1.5, -0.5) node[fill,cross,inner sep=3pt,rotate=-161.565,color=gray]{} -- (-4.5, -1.5);
\draw[-,gray] (-6.0, 2.5) node[fill,cross,inner sep=3pt,rotate=-9.462,color=gray]{} -- (6.0, 0.5);
\draw[-,gray] (-6.0, -1.5) node[fill,cross,inner sep=3pt,rotate=9.462,color=gray]{} -- (6.0, 0.5);
\draw[-,gray] (-3.0, 2.0) -- (6.0, 2.0);
\draw[-,gray] (0.0, 0.0) -- (6.0, 0.0);
\draw[-,gray] (-3.0, -1.0) -- (6.0, -1.0);
\draw[-,gray] (0.0, 1.0) -- (6.0, 1.0);
\draw[-,gray] (1.5, 0.5) -- (6.0, 0.5);
\draw[-,gray] (3.0, 0.0) -- (6.0, -0.5);
\draw[-,gray] (3.0, 1.0) -- (6.0, 1.5);
\draw[-,gray] (-0.0, -0.0) -- (4.5, -1.5);
\draw[-,gray] (-0.0, -0.0) -- (6.0, -0.0);
\draw[-,gray] (-3.0, 2.0) -- (-3.0, 2.5);
\draw[-,gray] (-0.0, -0.0) -- (6.0, 1.0);
\draw[-,gray] (-0.0, 1.0) -- (4.5, 2.5);
\draw[-,gray] (-0.0, 1.0) -- (6.0, 0.0);
\draw[-,gray] (-0.0, 1.0) -- (6.0, 1.0);
\draw[-,gray] (-3.0, -1.0) -- (-3.0, -1.5);
\fill[black!40!green] (5.8, 0.382) circle (2pt);
\draw[black!40!green,line width=1pt] (5.8, 0.382) -- (1.146, 0.382);
\draw[black!40!green,line width=1pt] (1.146, 0.382) -- (0.792, 0.736);
\draw[black!40!green,line width=1pt] (0.792, 0.736) -- (6, 0.736) node[above]{$8$};
\fill[black!40!green] (5.8, 0.382) circle (2pt);
\draw[black!40!green,line width=1pt] (5.8, 0.382) -- (0.0, 0.382);
\draw[black!40!green,line width=1pt] (0.0, 0.382) -- (-2.0, -0.618);
\draw[black!40!green,line width=1pt] (-2.82, -0.764) -- (6, -0.764) node[above]{$6$};
\draw (-2.7,0.5) node{$N_{4,2}^{\textup{trop}}=14$};
\end{tikzpicture}
\hspace{2mm}
\begin{tikzpicture}[scale=0.8,rotate=90]
\clip (-3,-1) rectangle (6.6,2);
\draw[dashed] (0.0, 0.5) -- (-1.0, 1.0);
\draw[dashed] (0.0, 0.5) -- (-1.0, 0.0);
\draw[dashed] (-1.5, -0.5) -- (-1.0, 0.0);
\draw[dashed] (-1.5, 1.5) -- (-1.0, 1.0);
\draw[dashed] (-1.5, 1.5) -- (-4.0, 2.0);
\draw[dashed] (-1.5, -0.5) -- (-4.0, -1.0);
\draw[dashed] (-6.0, -1.5) -- (-4.0, -1.0);
\draw[dashed] (-6.0, 2.5) -- (-4.0, 2.0);
\draw[-,gray] (0.0, 0.5) node[fill,cross,inner sep=3pt,rotate=90.0,color=gray]{} -- (0.0, 2.5);
\draw[-,gray] (0.0, 0.5) node[fill,cross,inner sep=3pt,rotate=-90.0,color=gray]{} -- (0.0, -1.5);
\draw[-,gray] (-1.5, 1.5) node[fill,cross,inner sep=3pt,rotate=-18.435,color=gray]{} -- (6.0, -1.0);
\draw[-,gray] (-1.5, -0.5) node[fill,cross,inner sep=3pt,rotate=18.435,color=gray]{} -- (6.0, 2.0);
\draw[-,gray] (-1.5, 1.5) node[fill,cross,inner sep=3pt,rotate=161.565,color=gray]{} -- (-4.5, 2.5);
\draw[-,gray] (-1.5, -0.5) node[fill,cross,inner sep=3pt,rotate=-161.565,color=gray]{} -- (-4.5, -1.5);
\draw[-,gray] (-6.0, 2.5) node[fill,cross,inner sep=3pt,rotate=-9.462,color=gray]{} -- (6.0, 0.5);
\draw[-,gray] (-6.0, -1.5) node[fill,cross,inner sep=3pt,rotate=9.462,color=gray]{} -- (6.0, 0.5);
\draw[-,gray] (-3.0, 2.0) -- (6.0, 2.0);
\draw[-,gray] (0.0, 0.0) -- (6.0, 0.0);
\draw[-,gray] (-3.0, -1.0) -- (6.0, -1.0);
\draw[-,gray] (0.0, 1.0) -- (6.0, 1.0);
\draw[-,gray] (1.5, 0.5) -- (6.0, 0.5);
\draw[-,gray] (3.0, 0.0) -- (6.0, -0.5);
\draw[-,gray] (3.0, 1.0) -- (6.0, 1.5);
\draw[-,gray] (-0.0, -0.0) -- (4.5, -1.5);
\draw[-,gray] (-0.0, -0.0) -- (6.0, -0.0);
\draw[-,gray] (-3.0, 2.0) -- (-3.0, 2.5);
\draw[-,gray] (-0.0, -0.0) -- (6.0, 1.0);
\draw[-,gray] (-0.0, 1.0) -- (4.5, 2.5);
\draw[-,gray] (-0.0, 1.0) -- (6.0, 0.0);
\draw[-,gray] (-0.0, 1.0) -- (6.0, 1.0);
\draw[-,gray] (-3.0, -1.0) -- (-3.0, -1.5);
\fill[black!40!green] (5.8, 0.382) circle (2pt);
\draw[black!40!green,line width=1pt] (5.8, 0.382) -- (1.146, 0.382);
\draw[black!40!green,line width=1pt] (1.146, 0.382) -- (-0.27, 1.09);
\draw[black!40!green,line width=1pt] (-0.27, 1.09) -- (6, 1.09) node[above]{$5$};
\draw (-2.7,0.5) node{$N_{5,1}^{\textup{trop}}=5$};
\end{tikzpicture}
\caption{Broken lines used to compute $N_{p,q}(\mathbb{P}^2,2L)$. The numbers above the infinite segments are the coefficients $a_{\mathfrak{b}}$.}
\label{fig:sage}
\end{figure}
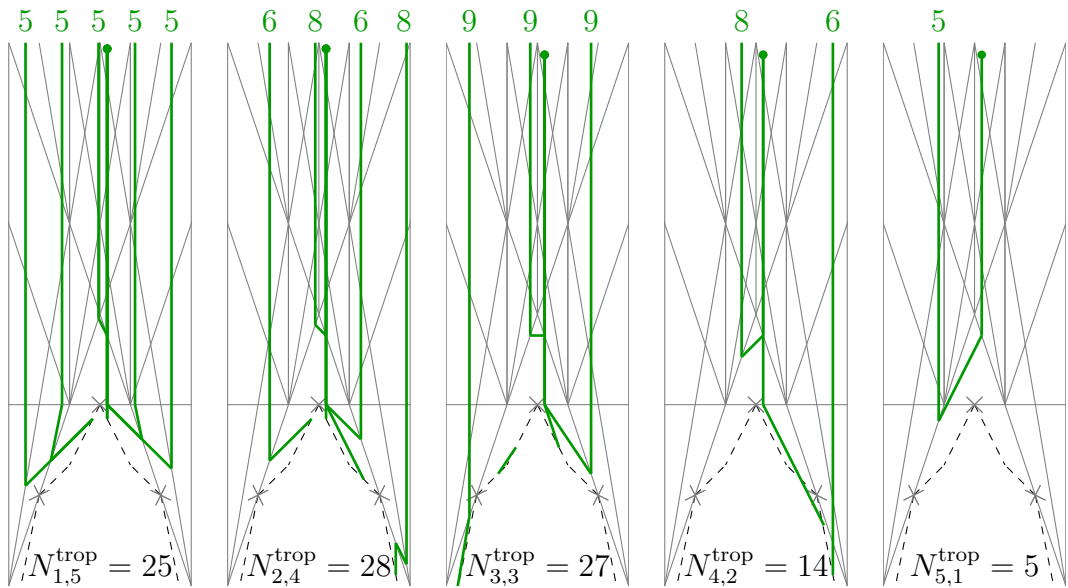



\end{document}